\documentclass[12pt, twoside]{article}
\usepackage{amsmath,amsthm,amssymb}
\usepackage{times}
\usepackage{enumerate}
\usepackage{bm}
\usepackage[all]{xy}
\usepackage{mathrsfs}
\usepackage{amscd}
\usepackage{color}
\def\titlerunning#1{\gdef\titrun{#1}}
\makeatletter
\def\author#1{\gdef\autrun{\def\and{\unskip, }#1}\gdef\@author{#1}}
\def\address#1{{\def\and{\\\hspace*{18pt}}\renewcommand{\thefootnote}{}%
		\footnote {#1}}%
	\markboth{\autrun}{\titrun}}
\makeatother
\def\email#1{e-mail: #1}

\def\keywords#1{\par\medskip
	\noindent\textbf{Keywords.} #1}

\newtheorem{theorem}{Theorem}[section]
\newtheorem{corollary}[theorem]{Corollary}
\newtheorem{lemma}[theorem]{Lemma}
\newtheorem{proposition}[theorem]{Proposition}
\theoremstyle{definition}
\newtheorem{definition}[theorem]{Definition}
\newtheorem{remark}[theorem]{Remark}
\newtheorem{example}[theorem]{Example}
\numberwithin{equation}{section}
\newtheorem{question}{Question}

\xyoption{all}

\def \C {\mathbb{C}}

\def \a {\alpha }
\def \b {\beta}

\def \de {\delta}
\def \De {\Delta}
\def \la {\lambda}
\def \La {\Lambda}
\def\w {\omega}
\def\Om{\Omega}

\def\pa{\partial}
\def\na {\nabla}
\def\Ga{\Gamma}

\begin{document}
	
\baselineskip=17pt

\titlerunning{Hodge decomposition and Hard Lefschetz Condition on almost K\"{a}hler manifolds}
\title{Hodge decomposition and Hard Lefschetz Condition on almost K\"{a}hler manifolds }

\author{Teng Huang and Weiwei Wang}

\date{}

\maketitle

\address{Teng Huang: School of Mathematical Sciences, CAS Key Laboratory of Wu Wen-Tsun Mathematics, University of Science and Technology of China, Hefei, Anhui, 230026, People’s Republic of China; \email{htmath@ustc.edu.cn;htustc@gmail.com}}
\address{Weiwei Wang: School of Mathematical Sciences, University of Electronic Science and Technology of China, Chengdu, Sichuan, 611731, People’s Republic of China; \email{wawnwg123@163.com}}
\begin{abstract}
In this article, we discuss the spaces of harmonic forms $\mathcal{H}^{\bullet}_{d}$ over a closed almost K\"{a}hler manifold $(X, J,\w)$.   We show that if the almost complex structure $J$ on the almost K\"{a}hler manifold $X$ is not too non-integrable in some sense, then the spaces $\mathcal{H}^{\bullet}_{d}$ have the Hodge decomposition  $\mathcal{H}^{k}_{d}=\oplus_{p+q=k}\mathcal{H}^{p,q}_{d}$. As a consequence, the not too non-integrable almost complex structure $J$ is complex $C^{\infty}$-pure-and-full, and the Hard Lefschetz Condition (HLC) on $\mathcal{H}^{\bullet}_{d}$ is satisfied. Moreover, we can prove a rigidity result for the closed $4$-dimensional almost K\"{a}hler manifold with $b^{+}_{2}(X)\geq2$.  
\end{abstract}
\section{Introduction}
Let $(X, J,\w)$ be a closed symplectic manifold of dimension $\dim X=2n$. There always exists an almost complex structure $J$ which is compatible with $\w$. This defines a metric $g$ by
$$g(\cdot,\cdot)=\w(\cdot,J\cdot)>0.$$
The triple $(X,J,\w)$ is called an almost K\"{a}hler manifold and $\w$ is called an almost K\"{a}hler form. If $J$ is integrable then it is K\"{a}hler.  

We will consider the   Hodge theory and the related spaces of harmonic forms on the almost K\"{a}hler manifold. The notion of symplectic Hodge theory was discussed in the late 1940s by Ehresmann and Libermann \cite{EL,Lib} and re-introduced by Brylinski \cite{Bry} about twenty years ago. Hodge theory plays an important role in Riemannian and complex geometry. But in symplectic geometry, its usefulness has been rather limited. 

On an almost K\"{a}hler manifold, the exterior differential of an almost complex manifold has four components, 
$$d=\pa+\mu+\bar{\pa}+\bar{\mu},$$ 
where $\mu$ and $\bar{\mu}$ arise from the Nijenhuis tensor and vanish if and only if the almost complex structure $J$ is integrable. Let $\de$ denote one of the components $\pa,\mu,\bar{\pa},\bar{\mu}$.  Recently, Cirici and Wilson \cite{CW1,CW2} defined the $\de$-Laplacian by letting
$$\De_{\de}:=\de\de^{\ast}+\de^{\ast}\de.$$
For all $p,q$, we will denote by
$$\mathcal{H}^{p,q}_{\de}:=\ker{\De_{\de}}\cap\Om^{p,q}=\ker{\de}\cap\ker{\de^{\ast}}\cap\Om^{p,q}$$
the space of $\de$-harmonic forms in bi-degree $(p,q)$. Here $\Om^{p,q}$ denotes the space of $(p,q)$-forms. By introducing a symplectic Hodge star operator $\ast_{s}$,  Brylinski \cite{Bry} proposed a Hodge theory of closed symplectic manifolds. The space of symplectic harmonic $k$-forms is $$\mathcal{H}^{k}_{sym}:=\ker{d}\cap\ker{d^{\La}}\cap\Om^{k},$$ 
where $d^{\La}:=[d,\La]=\ast_{s}d\ast_{s}$. He also showed that in almost K\"{a}hler manifold, a form of pure $(p,q)$ is in $\mathcal{H}^{p+q}_{sym}$ if only if it is in $\mathcal{H}^{p+q}_{d}$. This gives an inclusion $\bigoplus_{p+q=k}\mathcal{H}^{p,q}_{d}\hookrightarrow\mathcal{H}^{p+q}_{sym}$. In general, this is strict. Indeed, Yan \cite{Yan} showed that for $k=0,1,2$, every cohomology class has a symplectic harmonic representative.

In \cite{CW1}, Cirici and Wilson considered the spaces of $\bar{\pa}$-$\mu$-harmonic forms given by the intersections
$$\mathcal{H}^{p,q}_{\bar{\pa}}\cap\mathcal{H}^{p,q}_{\mu}.$$
These are identified with the kernel of the self-adjoint elliptic operator given by $\De_{\bar{\pa}}+\De_{\mu}$. Recall that $h^{p,q}=\dim (\mathcal{H}^{p,q}_{\bar{\pa}}\cap\mathcal{H}^{p,q}_{\mu})$ denotes the dimension of the space of $\bar{\pa}$-$\mu$-harmonic forms of type $(p,q)$. Let $b^{k}:=\dim H^{k}_{d}$ denote the Betti numbers of a closed manifold $X$. Theorem 4.1 in \cite{CW1} implies that, for a closed almost K\"{a}hler manifold, the numbers $h^{p,q}$ are bounded above by the Betti numbers and satisfy the Hodge diamond-type symmetries,
$$h^{p,q}=h^{q,p}=h^{n-p,n-q}=h^{n-q,n-p}.$$
There is an inclusion 
$$\bigoplus_{p+q=k}\mathcal{H}^{p,q}_{d}\subset\mathcal{H}^{k}_{d}$$
which gives topological bounds
$$\sum_{p+q=k}h^{p,q}\leq b^{k}, \forall\  k\geq0.$$
K\"{a}hler manifolds play a central role at the intersection of complex, symplectic and Riemannian geometry. It's well known that the spaces of harmonic forms of a closed K\"{a}hler manifold satisfy two crucial properties: Hodge decomposition and Hard Lefschetz Condition, which do not hold for a general closed symplectic manifold or complex manifold.
\begin{question}
It's natural to ask under what condition, $b^{k}=\sum_{p+q=k}h^{p,q}$ is established?
\end{question} 
For any $k\geq0$, there is a decomposition as follows (see Proposition \ref{P1}):
$$\Om^{k}=\ker(\De_{\bar{\pa}}+\De_{\mu})\cap\Om^{k}(X)\bigoplus \mathrm{Im}(\De_{\bar{\pa}}+\De_{\mu})\cap\Om^{k}(X).$$
One also can see (cf. Theorem \ref{T3}) $$\ker(\De_{\bar{\pa}}+\De_{\mu})\cap\Om^{k}(X)=\bigoplus_{p+q=k}\mathcal{H}_{d}^{p,q}.$$ 
Since $\De_{\bar{\pa}}+\De_{\mu}=\De_{\pa}+\De_{\bar{\mu}}$, for any $\a\in\Om^{\bullet}(X)$, we have
\begin{equation}\label{E1}
\begin{split}
\langle(\De_{\bar{\pa}}+\De_{\mu})\a,\a\rangle_{L^{2}}&=\frac{1}{2}\langle(\De_{\bar{\pa}}+\De_{\pa})\a,\a\rangle_{L^{2}}+\frac{1}{2}\langle(\De_{\mu}+\De_{\bar{\mu}})\a,\a\rangle_{L^{2}}\\
&\geq\frac{1}{2}\langle(\De_{\mu}+\De_{\bar{\mu}})\a,\a\rangle_{L^{2}}.\\
\end{split}
\end{equation}
We denote  $\mathcal{M}(k,c)$ the family of closed almost K\"{a}hler manifolds with 
$$\langle(\De_{\bar{\pa}}+\De_{\mu})\a,\a\rangle_{L^{2}}\geq c(\De_{\mu}+\De_{\bar{\mu}})\a,\a\rangle_{L^{2}}$$
for any $\a\in(\bigoplus_{p+q=k}\mathcal{H}_{d}^{p,q})^{\bot}$. According to Equation \ref{E1}, we may assume without loss of generality that $c\geq\frac{1}{2}$. Let's further suppose that $c$ is large enough, then for any closed almost K\"{a}hler manifolds in $\mathcal{M}(k,c)$, we will show  $b^{k}=\sum_{p+q=k}h^{p,q}$ holds.  In fact, by equations (\ref{E10}), one can see that $c>20$ is a sufficient condition in Theorem \ref{T1}.  
\begin{theorem}\label{T1}
Let $(X,J,\w)$ be a closed $2n$-dimensional almost K\"{a}hler manifold, let $0<k\leq n$. There exists an uniform positive constant $c>0$ such that if $X\in\mathcal{M}(k,c)$, then
$$\mathcal{H}^{k}_{d}(X)=\bigoplus_{p+q=k}\mathcal{H}^{p,q}_{d}(X).$$
\end{theorem}
\begin{remark}
Cirici and Wilson extended the Hodge and Serre dualities, hard Lefschetz duality, and Lefschetz decomposition to all spaces of $\De_{d}$-harmonic $(p,q)$-forms (see \cite{CW1}).  Under the setting of Theorem \ref{T1}, the space of $\De_{d}$-harmonic $k$-forms also has these properties.
\end{remark}
We say that $J$ is complex-$C^{\infty}$-pure-and-full at $k$-th stage if $J$ induces a decomposition of 
$$H^{k}_{dR}(X)=\bigoplus_{p+q=k}H^{p,q}(X).$$ 
For $k=2$, $J$ is called complex-$C^{\infty}$-pure-and-full (see \cite{AT,DLZ}). By Theorem \ref{T1} and Theorem \ref{T4}, we can prove that if the Nijenhuis tensor ensures that $X\in\mathcal{M}(k,c)$ for some constant $c$, then $J$ is complex-$C^{\infty}$-pure-and-full in $k$-th stage.
\begin{corollary}
	Let $(X,J,\w)$ be a closed $2n$-dimensional almost K\"{a}hler manifold, let $0<k\leq n$. There exists an uniform positive constant $c>0$ such that if $X\in\mathcal{M}(k,c)$, then
	$J$ is complex-$C^{\infty}$-pure-and-full in $k$-stage.
\end{corollary}
A special class of symplectic manifolds is represented by those satisfying the Hard Lefschetz Condition (HLC), i.e., those closed $2n$-dimensional symplectic manifolds $(X,\w)$ for which the map
$$L^{n-k}:H^{k}_{dR}(X)\rightarrow H^{2n-k}_{dR}(X),\ \forall\ 0\leq k<n$$
are isomorphisms. In particular, the following facts are equivalent on a closed symplectic  symplectic  $(X,\w)$ (cf. \cite{Bry,Mat,Mer,TT,Yan}) 
\begin{itemize}
	\item the hard Lefschetz condition (HLC) holds;
	\item the Brylinski conjecture, i.e. the existence of a symplectic harmonic form in each de Rham cohomology class;
	\item the $dd^{\La}$-lemma, i.e. every $d^{\La}$-closed, $d$-exact form is also $dd^{\La}$-exact.
\end{itemize}
A classical result states if $(X,\w,J)$ is a closed K\"{a}hler manifold, then $(X,\w)$ satisfies the HLC \cite{Huy} and the de Rham complex $(d,\Om^{\ast}(X))$ is a formal DGA in the sense of Sullivan \cite{DGMS}. We call the Hard Lefschetz Condition holds on $(\mathcal{H}^{\bullet}_{d},\w)$ if 
$$L^{n-k}:\mathcal{H}^{k}_{d}(X)\rightarrow\mathcal{H}^{2n-k}_{d}(X),\ \forall\ 0\leq k<n$$
are isomorphisms. In \cite{TW}, the authors studied the Hard Lefschetz property of $\ker\De_{d}$ on almost K\"{a}hler manifold. It is easy to see that the HLC on  $(\mathcal{H}^{\bullet}_{d},\w)$ implies the HLC on $(H^{\bullet}_{dR},\w)$ (see \cite{TW}).
\begin{question}
Whether HLC on $(H^{\bullet}_{dR},\w)$ is equivalent to HLC on $(\mathcal{H}^{\bullet}_{d},\w)$?
\end{question}
By Proposition \ref{P8}, one can see that HLC on $(\mathcal{H}^{\bullet}_{d},\w)$ is equivalent to $\mathcal{H}^{k}_{d}=\mathcal{H}^{k}_{d^{\La}}$. It's easy to see that the space $\oplus_{p+q=k}\mathcal{H}^{p,q}_{d}\subset\mathcal{H}^{k}_{d}\cap\mathcal{H}^{k}_{d^{\La}}$. We will study another elliptic operator $\De_{\bar{\pa}+\mu}$ on a closed almost K\"{a}hler manifold. For any $k\geq0$, there is a decomposition as follows:
$$\Om^{k}=\ker(\De_{\bar{\pa}+\mu})\cap\Om^{k}(X)\bigoplus \mathrm{Im}(\De_{\bar{\pa}+\mu})\cap\Om^{k}(X).$$
One can also see
$$\ker(\De_{\bar{\pa}+\mu})\cap\Om^{k}(X)=\mathcal{H}^{k}_{d}(X)\cap\mathcal{H}^{k}_{d^{\La}}(X).$$ 
Denote by $\mathcal{H}^{k}_{\bar{\pa}+\mu}(X)=\ker(\De_{\bar{\pa}+\mu})$.  We denote  $\widetilde{\mathcal{M}}(k,\tilde{c})$ the family of closed almost K\"{a}hler manifolds with 
$$(\De_{\bar{\pa}+\mu}\a,\a)_{L^{2}(X)}\geq \tilde{c}\langle(\De_{\mu}+\De_{\bar{\mu}})\a,\a\rangle_{L^{2}}$$
for any $\a\in(\mathcal{H}^{k}_{\bar{\pa}+\mu}(X))^{\bot}$. In our article, we will show that the spaces of harmonic forms $(\mathcal{H}_{d}^{k}(X),\w)$, $0<k<n$, satisfy HLC if the Nijenhuis tensor ensures that $X\in\widetilde{\mathcal{M}}(k,\tilde{c})$ for some positive constant $c$. 
\begin{theorem}\label{T8}
	Let $(X,J,\w)$ be a closed $2n$-dimensional almost K\"{a}hler manifold, let $0<k\leq n$. There exists an uniform positive constant $\tilde{c}>0$ such that if $X\in\widetilde{\mathcal{M}}(k,\tilde{c})$, then
	$$\mathcal{H}^{k}_{d}(X)=\mathcal{H}^{k}_{d^{\La}}(X).$$
\end{theorem}
Following Theorem \ref{T8} and Proposition \ref{P8}, we get
\begin{corollary}\label{C5}
Let $(X,J,\w)$ be a closed $2n$-dimensional almost K\"{a}hler manifold.
There exists an uniform positive constant $c>0$ such that if $X\in\cap_{k=1}^{n-1}\widetilde{\mathcal{M}}(k,\tilde{c})$, then $(\mathcal{H}_{d}^{\bullet}(X),\w)$ satisfies the hard Lefschetz condition.
\end{corollary}
\begin{remark}
We denote  $\overline{\mathcal{M}}(k,\bar{c})$ the family of closed almost K\"{a}hler manifolds with $$\langle\De_{d}\a,\a\rangle_{L^{2}}\geq \bar{c}\langle(\De_{\mu}+\De_{\bar{\mu}})\a,\a\rangle_{L^{2}}$$
for any $\a\in(\mathcal{H}^{k}_{d}(X))^{\bot}$. For large enough $c$, we have
$$\mathcal{M}(k,c)\subset\widetilde{\mathcal{M}}(k,\tilde{c})\subset\overline{\mathcal{M}}(k,\bar{c}),$$
where $\tilde{c}$ and $\bar{c}$ are two positive constants (see Proposition \ref{P12} and Proposition \ref{P11}). In \cite{dT}, the authors proved that if the Nijenhuis tensor ensures that $X\in\overline{\mathcal{M}}(k+1,\bar{c})$, then the map
$L_{\w}^{n-k}:H^{k}_{dR}(X)\rightarrow H^{2n-k}_{dR}(X)$ induces an isomorphic map in cohomology for some positive constant $\bar{c}$. 
\end{remark}
The last purpose of this article is that
\begin{question}
	whether a closed almost K\"{a}hler manifold $X$ is K\"{a}hlerian when the Nijenhuis tensor $N$ of $X$ is small enough in some sense?	
\end{question} 
For any closed $4$-dimensional almost K\"{a}hler manifold $X$, it's easy to see $b^{+}_{2}(X)\geq1$ (since $d\w=0$ and $d^{\ast}\w=0$). It is well known that there exist closed symplectic manifolds satisfying the HLC but not admitting any K\"{a}hlerian structure. For example, any non-K\"{a}hlerian simply connected symplectic $4$-manifold (or see \cite{GM}). However, under the assumption that the Nijenhuis tensor is small enough, we can prove the following rigidity result. We also refer to \cite{Wei} for a detailed study on the Calabi-Yau equation on $4$-dimensional almost K\"{a}hler manifolds under the condition that  Nijenhuis tensor $N$ is small enough in $L^{1}$-norm. 
\begin{theorem}\label{T2}
	Let $(X,J,\w)$ be a $4$-dimensional closed almost K\"{a}hler manifold with $b^{+}_{2}(X)\geq2$. There exists a uniform positive constant $c>0$ such that if $X\in\mathcal{M}(2,c)$, then the almost complex structure $J$ is integrable.	
\end{theorem}
\section{Preliminaries}
\subsection{Almost K\"{a}hler manifold }
We recall some definitions and results on the differential forms for almost complex manifolds. Let $(X,J)$ be a closed $2n$-dimensional almost Hermitian manifold and $J$ be a smooth almost complex structure on $X$. There is a natural action of $J$ on the space $\Om^{k}(X,\C):=\Om^{k}(X)\otimes\C$, which induces a topological type decomposition
$$\Om^{k}(X,\C)=\bigoplus_{p+q=k}\Om^{p,q}(X,\C),$$
where  $\Om^{p,q}(X,\C)$ denotes the space of complex forms of type $(p,q)$ with respect to $J$. As a matter of notation, bigraduation without further specification refers to complex forms. We have that
$$d:\Om^{p,q}\rightarrow\Om^{p+2,q-1}\oplus\Om^{p+1,q}\oplus\Om^{p,q+1}\oplus\Om^{p-1,q+2}$$
and so $d$ splits according as
$$d=\mu+\pa+\bar{\pa}+\bar{\mu},$$
where all the pieces are graded algebra derivations, $\mu$, $\bar{\mu}$ are  $0$-order differential operators. Note that each component of $d$ is a derivation, with bi-degrees given by 
$$|\mu|=(2,-1),\ |\pa|=(1,0),\ |\bar{\pa}|=(0,1),\ |\bar{\mu}|=(-1,2).$$
The integrability theorem of Newlander and Nirenberg states that the almost complex structure $J$ is integrable if and only if $N=0$, where
$$N:TX\otimes TX\rightarrow TX,$$
denotes the Nigenhuis tensor
$$N(X,Y):=[X,Y]+J[X,JY]+J[JX,Y]-[JX,JY].$$ 
In next lemma shows that $J$ is integrable if only if $N=0$, i.e, $\mu=0$ 
\begin{lemma}(\cite[Lemma 2.1]{CW2})\label{L7}
	With the above notation, we have:
	$$\mu+\bar{\mu}=-\frac{1}{4}(N\otimes Id_{\C})^{\ast},$$
	where the right hand side has been extended over all forms as a derivation.	
\end{lemma}
Expanding the equation $d^{2}=0$ we obtain the following set of equations:
\begin{equation*}
\begin{split}
&\mu^{2}=0,\\
&\mu\pa+\pa\mu=0,\\
&\pa^{2}+\mu\bar{\pa}+\bar{\pa}\mu=0,\\
&\pa\bar{\pa}+\bar{\pa}\pa+\mu\bar{\mu}+\bar{\mu}\mu=0,\\
&\bar{\pa}^{2}+\bar{\mu}\pa+\pa\bar{\mu}=0,\\
&\bar{\mu}\bar{\pa}+\bar{\pa}\bar{\mu}=0.\\
&\bar{\mu}^{2}=0.\\
\end{split}
\end{equation*} 
\begin{definition}
An almost K\"{a}hler structure on $2n$-dimensional manifold $X$ is a pair
$(\w,J)$, where $\w$ is a symplectic form,  and $J$ is an almost complex structure calibrated by $\w$.
\end{definition}
For any almost K\"{a}hler manifold $(X,J,\w)$ there is an associated Hodge-star operator \cite{Huy}
$\ast:\Om^{p,q}\rightarrow\Om^{n-q,n-p}$ defined by 
$$\a\wedge\ast\bar{\b}=\langle\a,\b\rangle\frac{\w^{n}}{n!}.$$
The Lefschetz operator  $L:\Om^{p,q}\rightarrow\Om^{p+1,q+1}$ defined by
$$L(\a)=\w\wedge\a.$$
It has adjoint $\La=\ast^{-1}L\ast$. There is a Lefschetz decomposition on complex $k$-forms
$$\Om^{k}(X,\C)=\bigoplus_{r\geq0}L^{r}P^{k-2r},$$
where $P^{\bullet}=\ker{\La}\cap\Om^{\bullet}(X,\mathbb{C})$. In local coordinates, the Nijenhuis tensor can be written as 
$$N_{jk}^{i}=J_{k}^{l}\pa_{l}J^{i}_{j}+J^{i}_{l}\pa_{j}J^{l}_{k}-J^{l}_{j}\pa_{l}J^{i}_{k}-J^{i}_{l}\pa_{k}J^{l}_{j}.$$ 
On an almost K\"{a}hler manifold, the Nigenhuis tensor can be written in the simpler form
$$N^{i}_{jk}=2(\na^{i}J^{l}_{j})J_{kl}.$$
\begin{lemma}(cf. \cite{Wei} and \cite{OS98,OS01})\label{L5}
For a $2n$-dimensional closed almost K\"{a}hler  manifold $(X,J,\w)$, we have
	$$|\na J|^{2}=\frac{1}{4}|N|^{2}=2(\tau^{\ast}-\tau),$$
	where $\tau^{\ast}$ (resp. $\tau$) is the $\ast$-scalar curvature (resp. scalar curvature). Furthermore,
	$$|\na J|^{2}\leq C(n)\|Rm(g)\|_{C^{0}},$$
	where $C(n)$ is a positive constant depending only on dimension, and $Rm(g)$ denote the Riemannian curvature tensor of an almost K\"{a}hler metric $g$.
\end{lemma}
\subsection{Almost K\"{a}hler identities}
In \cite{CW1,CW2}, the authors extended the K\"{a}hler identities to the non-integrable setting and deduced several geometric and topological consequences. There are some operators $\de=\mu, \pa, \bar{\pa}, \bar{\mu}$ on a smooth almost K\"{a}hler manifold $X$, and $\de$ have adjoint operators $\de^{\ast}$ when $X$ is closed. One can check that $\bar{\pa}^{\ast}=-\ast\pa\ast$, $\pa^{\ast}=-\bar{\pa}\ast$ and $\bar{\mu}^{\ast}=-\ast\mu\ast$, $\mu^{\ast}=-\ast\bar{\mu}\ast$. For any almost K\"{a}hler manifold, there is a $\mathbb{Z}_{2}$-graded Lie algebra of operators action on the $(p,q)$-forms, generated by eight odd operators $\bar{\pa},\pa,\bar{\mu},\mu,\bar{\pa}^{\ast},\pa^{\ast},\bar{\mu}^{\ast},\mu^{\ast}$ and even degree operators $L,\La,H$ (cf. \cite[Section 3]{CW1}).

We define the graded commutator of operators $A$ and $B$ by
$$[A,B]=AB-(-1)^{deg(A)deg(B)}BA$$
where $deg(A)$ denotes the total degree of $A$. On an almost almost K\"{a}hler manifold, Cirici-Wilson \cite{CW1} constructed some identities which called almost K\"{a}hler identities.
\begin{proposition}(\cite[Proposition 3.1]{CW1})\label{P2}
For any almost K\"{a}hler manifold the following identities hold:\\
(1) $[L,\bar{\mu}]=[L,\mu]=0$ and $[\La,\bar{\mu}^{\ast}]=[\La,\mu^{\ast}]=0$.\\
(2)	$[L,\bar{\pa}]=[L,\pa]=0$ and $[\La,\bar{\pa}^{\ast}]=[\La,\pa^{\ast}]=0$.\\
(3) $[L,\bar{\mu}^{\ast}]=\sqrt{-1}\mu$, $[L,\mu^{\ast}]=-\sqrt{-1}\bar{\mu}$ and   $[\La,\bar{\mu}]=\sqrt{-1}\mu^{\ast}$, $[\La,\mu]=-\sqrt{-1}\bar{\mu}^{\ast}$.\\
(4) $[L,\bar{\pa}^{\ast}]=-\sqrt{-1}\pa$, $[L,\pa^{\ast}]=\sqrt{-1}\bar{\pa}$ and   $[\La,\bar{\pa}]=-\sqrt{-1}\pa^{\ast}$, $[\La,\pa]=\sqrt{-1}\bar{\pa}^{\ast}$.
\end{proposition}
If $C$ is another operator of degree $deg(C)$, the following Jacobi identity is easy to check
$$(-1)^{deg(C)deg(A)}\big{[}A,[B,C]\big{]}+(-1)^{deg(A)deg(B)}\big{[}B,[C,A]\big{]}+(-1)^{deg(B)deg(C)}\big{[}C,[A,B]\big{]}=0.$$
\begin{proposition}(\cite[Proposition 3.2]{CW1})\label{P3}
For any almost K\"{a}hler manifold the following identities hold:\\
	(1) $[\bar{\mu},\mu^{\ast}]=[\mu,\bar{\mu}^{\ast}]=0$.\\
	(2)	$[\bar{\mu},\pa^{\ast}]=[\bar{\pa},\mu^{\ast}]$ and $[\mu,\bar{\pa}^{\ast}]=[\pa,\bar{\mu}^{\ast}]$.\\
	(3) $[\pa,\bar{\pa}^{\ast}]=[\bar{\mu}^{\ast},\bar{\pa}]+[\mu,\pa^{\ast}]$ and $[\bar{\pa},\pa^{\ast}]=[\mu^{\ast},\pa]+[\bar{\mu},\bar{\pa}^{\ast}]$.
\end{proposition}
In \cite{CW1}, they also gave several relations concerning various Laplacians.
\begin{proposition}(\cite[Proposition 3.3]{CW1})\label{P4}
For any almost K\"{a}hler manifold the following identities hold:\\
(1) $\De_{\bar{\mu}+\mu}=\De_{\bar{\mu}}+\De_{\mu}$.\\
(2)	$\De_{\bar{\pa}}+\De_{\mu}=\De_{\pa}+\De_{\bar{\mu}}.$\\
(3) $\De_{d}=2(\De_{\bar{\pa}}+\De_{\mu}+[\bar{\mu},\pa^{\ast}]+[\mu,\bar{\pa}^{\ast}]+[\pa,\bar{\pa}^{\ast}]+[\bar{\pa},\pa^{\ast}])$.
\end{proposition}
\subsection{Cohomology on almost K\"{a}hler manifold}
We define the second differential operator 
$$d^{\La}:=[d,\La]=d\La-\La d$$
on a smooth almost K\"{a}hler manifold. Following \cite[Lemma 2.9]{TY1}, we have
$$d^{\La}=-\ast J^{-1}dJ\ast.$$
where $J=\sum_{p,q}(\sqrt{-1})^{p-q}\Pi^{p,q}$ projects a $k$-form onto its $(p,q)$-part times the multiplicative factor $(\sqrt{-1})^{p-q}$. For the adjoint operator of $d^{\La}$ with respect to Hodge star $\ast$, we have 
$$d^{\La\ast}=J^{-1}dJ=[L,d^{\ast}].$$
Since $(d^{\La})^{2}=0$, there is a natural cohomology
$$H^{k}_{d^{\La}}(X)=\frac{\ker d^{\La}\cap\Om^{k} }{\mathrm{Im}d^{\La}\cap\Om^{k} }.$$
This cohomology has been studied in \cite{Bry,Mat,Yan}. In \cite{Bry}, the author shown that $H^{k}_{d}(X)$ is isomorphism to $H^{2n-k}_{d^{\La}}(X)$ through the symplectic $\ast_{s}$ operator. We denote
$$\De_{d^{\La}}=d^{\La}d^{\La\ast}+d^{\La\ast}d^{\La}$$
by the $d^{\La}$ Laplacian operator. Applying elliptic theory, we get
$$H^{k}_{d^{\La}}(X)\cong\mathcal{H}^{k}_{d^{\La}}(X),$$
where 
$$\mathcal{H}^{k}_{d^{\La}}:=\{\a\in\Om^{k}(X):\De_{d^{\La}}\a=0  \}$$
is the space of $d^{\La}$-harmonic $k$-forms.

There are another two cohomologies which called $d+d^{\La}$- and $dd^{\La}$-cohomology on closed almost K\"{a}hler manifold (see \cite{TY1,TY2,TY3}).\\ 
(1) $d+d^{\La}$-cohomology: $$H^{k}_{d+d^{\La}}(X)=\frac{\ker(d+d^{\La})\cap\Om^{k}(X)  }{\mathrm{Im}dd^{\La}\cap\Om^{k}(X) }.$$
(2) $dd^{\La}$-cohomolgy:
$$H^{k}_{dd^{\La}}(X)=\frac{\ker(dd^{\La})\cap\Om^{k}(X)  }{(\mathrm{Im}d+\mathrm{Im}d^{\La})\cap\Om^{k}(X) }.$$
We then get 
$$H^{k}_{d+d^{\La}}(X)\cong\mathcal{H}^{k}_{d+d^{\La}}(X),\ H^{k}_{dd^{\La}}(X)\cong\mathcal{H}^{k}_{dd^{\La}}(X),$$
and 
$$H^{k}_{d+d^{\La}}(X)\cong H^{2n-k}_{dd^{\La}}(X),$$
where 
$$\mathcal{H}^{k}_{d+d^{\La}}(X)=\{\a\in\Om^{k}(X):(dd^{\La})^{\ast}\a=0,\ and\ d\a=d^{\La}\a=0 \}$$
is the space of $d+d^{\La}$-harmonic $k$-forms and 
$$\mathcal{H}^{k}_{dd^{\La}}(X)=\{\a\in\Om^{k}(X):(dd^{\La})\a=0,\ and\ d^{\ast}\a=d^{\La\ast}\a=0 \}$$
is the space of $dd^{\La}$-harmonic $k$-forms. If $(X,J,\w)$ satisfies $dd^{\La}$-Lemma, then $\dim\mathcal{H}^{k}_{d+d^{\La}}(X)=\dim\mathcal{H}^{k}_{d}(X)$ for all $k$. Furthermore, if $X$ is K\"{a}hlerian, i.e., $J$ is integrable, then one can see that $\mathcal{H}^{k}_{d+d^{\La}}(X)=\mathcal{H}^{k}_{d}(X)$. 
\begin{proposition}\label{P6}(cf. \cite[Proposition 5.1]{ST})
Let $(X,J,\w)$ be a closed $2n$-dimensional  almost K\"{a}hler manifold. We then have $$\mathcal{H}^{1}_{d+d^{\La}}=\mathcal{H}^{1}_{d}.$$
\end{proposition}
\begin{proof}
Let $\a\in\mathcal{H}^{1}_{d}(X)$, i.e., $d\a=0$ and $d^{\ast}\a=0$. Noting that $d^{\La}\a=[d,\La]\a=0$, we therefore have $\mathcal{H}^{1}_{d}\subset\mathcal{H}^{1}_{d+d^{\La}}$.  On the other hand, let $\a\in\mathcal{H}^{1}_{d+d^{\La}}(X)$, i.e., $d\a=0$ and $(dd^{\La})^{\ast}\a=0$. Then by the Hodge decomposition, there exists a function $f$ and a harmonic $1$-form such that 
$$\a=df+\a_{h}.$$
By the decomposition of $d^{\La}$-cohomology, we get 
$$f=d^{\La}\b+c$$
where $c$ is a constant and $\b$ is a $1$-form. Therefore,
$$0=(dd^{\La})^{\ast}\a=(dd^{\La})^{\ast}dd^{\La}\b\Longrightarrow dd^{\La}\b=0,$$
i.e., $df=0$. It implies that $\mathcal{H}^{1}_{d+d^{\La}}\subset\mathcal{H}^{1}_{d}$.
\end{proof}

\section{Hodge decomposition on almost K\"{a}hler manifold }
\subsection{An estimate on $\ker(\De_{\bar{\pa}}+\De_{\mu})$}
For closed almost K\"{a}hler manifolds,  we have (see Proposition \ref{P4})
\begin{equation}\label{E7}
\De_{\bar{\pa}}+\De_{\mu}=\De_{\pa}+\De_{\bar{\mu}}.
\end{equation}
Hence, Cirici and Wilson proved:
\begin{theorem}(\cite[Theorem 4.1]{CW1})\label{T3}
For any closed almost K\"{a}hler manifold of dimension $2n$ and for all $(p,q)$, we have identities
$$\mathcal{H}^{p,q}_{d}=\mathcal{H}^{p,q}_{\bar{\pa}}\cap\mathcal{H}^{p,q}_{\mu}=\mathcal{H}^{p,q}_{\pa}\cap\mathcal{H}^{p,q}_{\bar{\mu}}.$$	
Furthermore, 
$$\mathcal{H}^{k}_{\bar{\pa}}\cap\mathcal{H}^{k}_{\mu}=\bigoplus_{p+q=k}\mathcal{H}^{p,q}_{d}.$$
\end{theorem}
It is well known that for a closed complex manifold with Hermitian metric, Dolbeault cohomology is isomorphic to the space of $\bar{\pa}$-harmonic forms, defined as the kernel of the $\bar{\pa}$-Laplacian $\De_{\bar{\pa}}:=\bar{\pa}\bar{\pa}^{\ast}+\bar{\pa}^{\ast}\bar{\pa}$. In the almost K\"{a}hler setting, we have
	\begin{proposition}(cf. \cite{CW2})\label{P1}
	Let $(X,J,\w)$ be a closed almost K\"{a}hler manifold. We have a decomposition as follows:
	$$\Om^{k}(X)=\ker(\De_{\bar{\pa}}+\De_{\mu})\cap\Om^{k}(X)\bigoplus \mathrm{Im}(\De_{\bar{\pa}}+\De_{\mu})\cap\Om^{k}(X).$$
	In particular,
	$$\Om^{k}(X)=\bigoplus_{p+q=k}\mathcal{H}^{p,q}_{d}(X)\bigoplus \mathrm{Im}(\De_{\bar{\pa}}+\De_{\mu})\cap\Om^{k}(X).$$
\end{proposition}
\begin{proof}
	By (\ref{E7}),  we have
	\begin{equation*}
	\begin{split}
	\De_{\bar{\pa}}+\De_{\mu}&=\frac{1}{2}\left(\De_{\bar{\pa}}+\De_{\mu}\right)+\frac{1}{2}\left(\De_{\pa}+\De_{\bar{\mu}}\right)\\
	&\simeq \frac{1}{2}\left(\bar{\pa}\bar{\pa}^{\ast}+\bar{\pa}^{\ast}\bar{\pa}\right)+\frac{1}{2}\left(\pa\pa^{\ast}+\pa^{\ast}\pa\right)\\
	&\simeq \frac{1}{2}\De_{d}.
	\end{split}
	\end{equation*}
	where $\simeq$ denotes equivalence under symbol calculation.
	Hence, the operator $\De_{\bar{\pa}}+\De_{\mu}:\Om^{k}(X)\rightarrow\Om^{k}(X)$ is a 2-order elliptic operator.  Since this operator is also self-adjoint, we can get the decomposition (cf. \cite{Dem})
	$$\Om^{k}(X)=\ker(\De_{\bar{\pa}}+\De_{\mu})\cap\Om^{k}(X)\bigoplus \mathrm{Im}(\De_{\bar{\pa}}+\De_{\mu})\cap\Om^{k}(X).$$
	Furthermore, the relation $\ker(\De_{\bar{\pa}}+\De_{\mu})\cap\Om^{k}=\ker(\De_{d})\cap\oplus_{p+q=k}\Om^{p,q}$ gives that
	$$\Om^{k}(X)=\ker(\De_{d})\cap\oplus_{p+q=k}\Om^{p,q}(X)\bigoplus \mathrm{Im}(\De_{\bar{\pa}}+\De_{\mu})\cap\Om^{k}(X).$$
\end{proof}
\begin{remark}
By Proposition \ref{P1}, for any $\a^{k}\in\Om^{k}(X)$, we can write
	\begin{equation*}
	\a^{k}=\sum_{p+q=k}(\a^{p,q}_{h}+\b^{p,q}),
	\end{equation*}
	where $\a^{p,q}_{h}\in\ker\De_{d}$ and $\b^{p,q}\in(\mathcal{H}^{p,q}_{d})^{\bot}$.
	We denote $\b^{k}:=\sum_{p+q=k}\b^{p,q}$. Noting that $\oplus_{p+q=k}\mathcal{H}^{p,q}_{d}\subset\mathcal{H}^{k}_{d}$. Then $\b^{k}=(\De_{\bar{\pa}}+\De_{\mu})\tilde{\b}^{k}$ may be not in $(\mathcal{H}^{k}_{d})^{\bot}$. Since for any $\gamma\in\mathcal{H}^{k}_{d}$, we get
	$$\langle\b^{k},\gamma\rangle_{L^{2}}=\langle(\De_{\bar{\pa}}+\De_{\mu})\tilde{\b}^{k},\gamma\rangle_{L^{2}}=\langle\tilde{\b}^{k},(\De_{\bar{\pa}}+\De_{\mu})\gamma\rangle_{L^{2}}.$$
	But $(\De_{\bar{\pa}}+\De_{\mu})\gamma$ may be not zero.
\end{remark} 
\begin{proposition}\label{P9}
Suppose that the closed almost K\"{a}hler manifold $X$ is in $\mathcal{M}(k,c)$.  Then for any $\a^{k}:=\a^{k}_{h}+\b^{k}\in\Om^{k}(X)$ 
with $\a^{k}_{h}\in\oplus_{p+q=k}\mathcal{H}^{p,q}_{d}(X)$ and $\b^{k}\in(\oplus_{p+q=k}\mathcal{H}^{p,q}_{d}(X))^{\perp}$, we have
\begin{equation}\label{E4}
\langle\De_{d}\a^{k},\a^{k}\rangle_{L^{2}}\geq c_{2}(c-c_{1})\langle(\De_{\mu}+\De_{\bar{\mu}})\b^{k},\b^{k}\rangle_{L^{2}},
\end{equation}	
where $c_{1}, c_{2}$ are two uniform positive constants. In particular, if $\gamma^{k}\in(\oplus_{p+q=k}\mathcal{H}^{p,q}_{d}(X))^{\perp}$, then
\begin{equation}\label{E8}
\langle\De_{d}\gamma^{k},\gamma^{k}\rangle_{L^{2}}\geq c_{2}(c-c_{1})\langle(\De_{\mu}+\De_{\bar{\mu}})\gamma^{k},\gamma^{k}\rangle_{L^{2}}.
\end{equation}
\end{proposition}
	\begin{proof}
		Following identities in Proposition \ref{P3} and \ref{P4}, the Laplacian $\De_{d}$ can be written as
		\begin{equation*}
		\begin{split}
		\De_{d}=&2\De_{\bar{\pa}}+2\De_{\mu}+2\left([\bar{\mu},\pa^{\ast}]+[\mu,\bar{\pa}^{\ast}]+[\bar{\mu}^{\ast},\bar{\pa}]+[\mu,\pa^{\ast}]+[\mu^{\ast},\pa]+[\bar{\mu},\bar{\pa}^{\ast}]\right).
		\end{split}
		\end{equation*}
		Since $\a^{k}_{h}\in\oplus_{p+q=k}\mathcal{H}^{p,q}_{d}(X)$, by Theorem \ref{T3}, we have
		$$\mu \a^{k}_{h}=\mu^{\ast} \a^{k}_{h}= \bar{\mu}\a^{k}_{h}=\bar{\mu}^{\ast} \a^{k}_{h}=0,$$
		and
		$$\pa \a^{k}_{h}=\pa^{\ast} \a^{k}_{h}= \bar{\pa}\a^{k}_{h}=\bar{\pa}^{\ast} \a^{k}_{h}=0.$$
		Noting that $\a^{k}=\a^{k}_{h}+\b^{k}$ with $\a^{k}_{h}\in\oplus_{p+q=k}\mathcal{H}^{p,q}_{d}(X)$ and $\b^{k}\in(\oplus_{p+q=k}\mathcal{H}^{p,q}_{d}(X))^{\perp}$, the interior product of $\De_{d}\a^{k}$ and $\a^{k}$ is 
		\begin{equation*}
		\begin{split}
		\langle\De_{d}\a^{k},\a^{k}\rangle_{L^{2}}=&2\langle(\De_{\bar{\pa}}+\De_{\mu})\b^{k},\b^{k}\rangle_{L^{2}}\\
		&+2\langle[\bar{\mu},\pa^{\ast}]\b^{k},\b^{k}\rangle_{L^{2}}+2\langle[\mu,\bar{\pa}^{\ast}]\b^{k},\b^{k}\rangle_{L^{2}}\\
		&+2\langle[\bar{\mu}^{\ast},\bar{\pa}]\b^{k},\b^{k}\rangle_{L^{2}}+2\langle[\mu,\pa^{\ast}]\b^{k},\b^{k}\rangle_{L^{2}}\\
		&+2\langle[\mu^{\ast},\pa]\b^{k},\b^{k}\rangle_{L^{2}}+2\langle[\bar{\mu},\bar{\pa}^{\ast}]\b^{k},\b^{k}\rangle_{L^{2}}.\\
		\end{split}
		\end{equation*}
		Here we use the fact 
		\begin{equation*}
		\begin{split}
		\langle[\bar{\mu},\pa^{\ast}]\a^{k},\a^{k}\rangle_{L^{2}}&=\langle\pa^{\ast}\a^{k},\bar{\mu}^{\ast}\a^{k}\rangle_{L^{2}}+\langle\bar{\mu}\a^{k},\pa\a^{k}\rangle_{L^{2}}\\
		&=\langle\pa^{\ast}\b^{k},\bar{\mu}^{\ast}\b^{k}\rangle_{L^{2}}+\langle\bar{\mu}\b^{k},\pa\b^{k}\rangle_{L^{2}}.\\
		\end{split}
		\end{equation*}
		The proof of other cases is in the same way. From the AM-GM Inequality, we conclude that
		\begin{equation*}
		\begin{split}
		2|\langle[\bar{\mu},\pa^{\ast}]\b^{k},\b^{k}\rangle_{L^{2}}|
		&=|\langle\pa^{\ast}\b^{k},\bar{\mu}^{\ast}\b^{k}\rangle_{L^{2}}+\langle\bar{\mu}\b^{k},\pa\b^{k}\rangle_{L^{2}}|\\
		&\leq2\|\pa^{\ast}\b^{k}\|\cdot\|\bar{\mu}^{\ast}\b^{k}\|+2\|\bar{\mu}\b^{k}\|\cdot\|\pa\b^{k}\|\\
		&\leq\varepsilon(\|\pa^{\ast}\b^{k}\|^{2}+\|\pa\b^{k}\|^{2})+\frac{1}{\varepsilon}(\|\bar{\mu}^{\ast}\b^{k}\|^{2}+\|\bar{\mu}\b^{k}\|^{2}),\\
		\end{split}
		\end{equation*}
		where $\varepsilon$ is a positive constant. In a similar way, we also have
		\begin{equation*}
		\begin{split}
		&2|\langle[\mu,\bar{\pa}^{\ast}]\b^{k},\b^{k}\rangle_{L^{2}}|\leq\varepsilon(\|\bar{\pa}^{\ast}\b^{k}\|^{2}+\|\bar{\pa}\b^{k}\|^{2})+\frac{1}{\varepsilon}(\|\mu^{\ast}\b^{k}\|^{2}+\|\mu\b^{k}\|^{2}),\\
		&2|\langle[\bar{\mu}^{\ast},\bar{\pa}]\b^{k},\b^{k}\rangle_{L^{2}}|\leq\varepsilon(\|\bar{\pa}^{\ast}\b^{k}\|^{2}+\|\bar{\pa}\b^{k}\|^{2})+\frac{1}{\varepsilon}(\|\bar{\mu}^{\ast}\b^{k}\|^{2}+\|\bar{\mu}\b^{k}\|^{2}),\\
		&2|\langle[\mu,\pa^{\ast}]\b^{k},\b^{k}\rangle_{L^{2}}|\leq\varepsilon(\|\pa^{\ast}\b^{k}\|^{2}+\|\pa\b^{k}\|^{2})+\frac{1}{\varepsilon}(\|\mu^{\ast}\b^{k}\|^{2}+\|\mu\b^{k}\|^{2})\\
		&2|\langle[\mu^{\ast},\pa]\b^{k},\b^{k}\rangle_{L^{2}}|\leq\varepsilon(\|\pa^{\ast}\b^{k}\|^{2}+\|\pa\b^{k}\|^{2})+\frac{1}{\varepsilon}(\|\mu^{\ast}\b^{k}\|^{2}+\|\mu\b^{k}\|^{2})\\
		&2|\langle[\bar{\mu},\bar{\pa}^{\ast}]\b^{k},\b^{k}\rangle_{L^{2}}|\leq\varepsilon(\|\bar{\pa}^{\ast}\b^{k}\|^{2}+\|\bar{\pa}\b^{k}\|^{2})+\frac{1}{\varepsilon}(\|\bar{\mu}^{\ast}\b^{k}\|^{2}+\|\bar{\mu}\b^{k}\|^{2}).\\
		\end{split}
		\end{equation*}
		Combining the above inequalities, we get
		\begin{equation*}
		\begin{split}
		\langle\De_{d}\a^{k},\a^{k}\rangle_{L^{2}}\geq&2\langle(\De_{\bar{\pa}}+\De_{\mu})\b^{k},\b^{k}\rangle_{L^{2}}\\
		&-2|\langle[\bar{\mu},\pa^{\ast}]\b^{k},\b^{k}\rangle_{L^{2}}|-2|\langle[\mu,\bar{\pa}^{\ast}]\b^{k},\b^{k}\rangle_{L^{2}}|\\
		&-2|\langle[\bar{\mu}^{\ast},\bar{\pa}]\b^{k},\b^{k}\rangle_{L^{2}}|-2|\langle[\mu,\pa^{\ast}]\b^{k},\b^{k}\rangle_{L^{2}}|\\
		&-2|\langle[\mu^{\ast},\pa]\b^{k},\b^{k}\rangle_{L^{2}}|-2|\langle[\bar{\mu},\bar{\pa}^{\ast}]\b^{k},\b^{k}\rangle_{L^{2}}|\\
		\geq&\langle(\De_{\bar{\pa}}+\De_{\pa})\b^{k},\b^{k}\rangle_{L^{2}}+\langle(\De_{\mu}+\De_{\bar{\mu}})\b^{k},\b^{k}\rangle_{L^{2}}\\
		&-\frac{3}{\varepsilon}\langle(\De_{\mu}+\De_{\bar{\mu}})\b^{k},\b^{k}\rangle_{L^{2}}-3\varepsilon\langle(\De_{\bar{\pa}}+\De_{\pa})\b^{k},\b^{k}\rangle_{L^{2}}\\
		=&(1-3\varepsilon)\langle(\De_{\bar{\pa}}+\De_{\pa}+\De_{\mu}+\De_{\bar{\mu}})\b^{k},\b^{k}\rangle_{L^{2}}+(3\varepsilon-\frac{3}{\varepsilon})\langle(\De_{\mu}+\De_{\bar{\mu}})\b^{k},\b^{k}\rangle_{L^{2}}\\
		\geq&2c(1-3\varepsilon)\langle(\De_{\mu}+\De_{\bar{\mu}})\b^{k},\b^{k}\rangle_{L^{2}}+(3\varepsilon-\frac{3}{\varepsilon})\langle(\De_{\mu}+\De_{\bar{\mu}})\b^{k},\b^{k}\rangle_{L^{2}}.
		\end{split}
		\end{equation*}
		Let $\varepsilon=\frac{1}{9}$, then 
		\begin{equation}\label{E10}
		\langle\De_{d}\a^{k},\a^{k}\rangle_{L^{2}}\geq\frac{4}{3}(c-20)\langle(\De_{\mu}+\De_{\bar{\mu}})\b^{k},\b^{k}\rangle_{L^{2}}.
		\end{equation}
We complete the proof by taking $c_1=20$ and  $c_2=\frac{4}{3}$.
\end{proof}
The Laplacian $\De_{\bar{\pa}+\mu}$ can be written as
\begin{equation*}
\De_{\bar{\pa}+\mu}=\De_{\bar{\pa}}+\De_{\mu}+[\bar{\pa},\mu^{\ast}]+[\mu,\bar{\pa}^{\ast}].
\end{equation*}
In a similar way of Proposition \ref{P9}, we prove
\begin{proposition}\label{P12}
Suppose that the closed almost K\"{a}hler manifold $X$ is in $\mathcal{M}(k,c)$. Then for any $\gamma^{k}\in(\oplus_{p+q=k}\mathcal{H}^{p,q}_{d}(X))^{\perp}$, we have
\begin{equation}\label{E14}
\langle\De_{\bar{\pa}+\mu}\gamma^{k},\gamma^{k}\rangle_{L^{2}}\geq c_{4}(c-c_{3})\langle(\De_{\mu}+\De_{\bar{\mu}})\gamma^{k},\gamma^{k}\rangle_{L^{2}},
\end{equation}
where $c_{3}, c_{4}$ are two uniform positive constants.
\end{proposition}
Noting that $(\mathcal{H}^{k}_{\bar{\pa}+\mu}(X))^{\perp}\subset(\oplus_{p+q=k}\mathcal{H}^{p,q}_{d}(X))^{\perp}$. Hence when $X\in\mathcal{M}(k,c)$ and $c$ is large enough, then $\mathcal{M}(k,c)\subset\widetilde{\mathcal{M}}(k,\tilde{c})$.

\subsection{Hodge decomposition on harmonic forms}
Before proving the Theorem \ref{T1}, we first observe a useful lemma as follows.
\begin{lemma}\label{L4}
	For any closed almost K\"{a}hler manifold of dimension $2n$ and for all $k$, we have
	\begin{equation}
	\mathcal{H}^{k}_{d}(X)\cap\mathcal{H}^{k}_{\mu}(X)\cap\mathcal{H}^{k}_{\bar{\mu}}(X)=\bigoplus_{p+q=k}\mathcal{H}^{p,q}_{d}(X).
	\end{equation}
\end{lemma}
\begin{proof}
	Let $\a^{k}\in\mathcal{H}^{k}_{d}\cap\mathcal{H}^{k}_{\mu}\cap\mathcal{H}^{k}_{\bar{\mu}}$, then $\De_{d}\a^{k}=\De_{\mu}\a^{k}=\De_{\bar{\mu}}\a^{k}=0$. This implies that $\mu \a=\mu^{\ast} \a= \bar{\mu}\a=\bar{\mu}^{\ast} \a=0$. From the identities in Proposition \ref{P3}  and  Proposition \ref{P4}, the Laplacian $\De_{d}$ is of  the form 
	\begin{equation*}
	\begin{split}
	\De_{d}&=\De_{\bar{\pa}}+\De_{\pa}+\De_{\mu}+\De_{\bar{\mu}}+2\left([\bar{\mu},\pa^{\ast}]+[\mu,\bar{\pa}^{\ast}]+[\pa,\bar{\pa}^{\ast}]+[\bar{\pa},\pa^{\ast}]\right)\\
	&=\De_{\bar{\pa}}+\De_{\pa}+\De_{\mu}+\De_{\bar{\mu}}\\
	&+2\left([\bar{\mu},\pa^{\ast}]+[\mu,\bar{\pa}^{\ast}]+[\bar{\mu}^{\ast},\bar{\pa}]+[\mu,\pa^{\ast}]+[\mu^{\ast},\pa]+[\bar{\mu},\bar{\pa}^{\ast}]\right).
	\end{split}
	\end{equation*}
	Noting that
	$$\langle[\bar{\mu},\pa^{\ast}]\a^{k},\a^{k}\rangle_{L^{2}}=\langle\pa^{\ast}\a^{k},\bar{\mu}^{\ast}\a^{k}\rangle_{L^{2}}+\langle\bar{\mu}\a^{k},\pa\a^{k}\rangle_{L^{2}}=0$$
	and similarly
	$$\langle [\mu,\bar{\pa}^{\ast}]\a,\a\rangle=\langle [\bar{\mu}^{\ast},\bar{\pa}]\a,\a\rangle=\langle[\mu,\pa^{\ast}]\a,\a\rangle=\langle[\mu^{\ast},\pa]\a,\a\rangle=\langle[\bar{\mu},\bar{\pa}^{\ast}]\a,\a\rangle=0.$$
	Hence,
	\begin{equation*}
	0=\langle\De_{d}\a^{k},\a^{k}\rangle_{L^{2}}=\langle(\De_{\bar{\pa}}+\De_{\pa})\a^{k},\a^{k}\rangle_{L^{2}}.
	\end{equation*}
	Therefore, $\De_{\bar{\pa}}\a^{k}=\De_{\pa}\a^{k}=0$, i.e., $$\mathcal{H}^{k}_{d}\cap\mathcal{H}^{k}_{\mu}\cap\mathcal{H}^{k}_{\bar{\mu}}\subset\bigoplus_{p+q=k}\mathcal{H}^{p,q}_{d}.$$
	The other side of inclusion follows trivially from Theorem \ref{T3}.
\end{proof}
\begin{proof}[\textbf{Proof of Theorem \ref{T1}}]
We denote $\a$ is a $\De_{d}$-harmonic $k$-form. Following Proposition \ref{P1}, we have following decompositions
$$\a=\a_{h}+\b,$$
where $\a_{h}\in\oplus_{p+q=k}\mathcal{H}^{p,q}_{d}$ and $\b:=(\De_{\bar{\pa}}+\De_{\mu})\gamma\in(\oplus_{p+q=k}\mathcal{H}^{p,q}_{d})^{\perp}$.
Noting that $\a_{h}$ is also in $\mathcal{H}^{k}_{d}(X)$. Hence $\De_{d}\b=0$. Following Proposition \ref{P9}, we then have 
$$(\De_{\bar{\mu}}+\De_{\mu})\b=0.$$
Hence by Lemma \ref{L4}, $\b$ is also in $\oplus_{p+q=k}\mathcal{H}^{p,q}_{d}$. Therefore, $\b=0$. It implies that $\a=\a_{h}$ i.e., $\mathcal{H}^{k}_{d}\subset\oplus_{p+q=k}\mathcal{H}^{p,q}_{d}$. The other side of inclusion is trivial. Therefore, $\mathcal{H}^{k}_{d}=\oplus_{p+q=k}\mathcal{H}^{p,q}_{d}$.
\end{proof}
Let $0<k\leq n$, for any  $\a\in(\bigoplus_{p+q=k}\mathcal{H}_{d}^{p,q})^{\bot}$,  we have
$$\langle(\De_{\bar{\pa}}+\De_{\mu})\a,\a\rangle_{L^{2}}\geq\la_{k}\langle\a,\a\rangle_{L^{2}},$$
where $\la_{k}$ is the smallest positive eigenvalue of $(\De_{\bar{\pa}}+\De_{\mu})$ acting on $k$-forms. 

In odd degrees, by the Hodge diamond-type symmetries, we have
$$\sum_{p+q=2k+1}h^{p,q}=2\sum_{0\leq p\leq k}h^{p,2k+1-p}.$$
Furthermore, from Theorem \ref{T1}, we also have
\begin{corollary}
Let $(X,J,\w)$ be a closed $2n$-dimensional almost K\"{a}hler manifold. If the Betti number $b^{2k+1}$ is odd, $0\leq k<n$, then there is a uniform positive constant $c(n)$ such that 
$$\la_{2k+1}\leq c(n)\|Rm(g)\|_{C^{0}}.$$
\end{corollary} 
\begin{proof}
By Lemma \ref{L5} and Lemma \ref{L7}, there exists a constant $C(n)$ such that 
$$  \langle(\De_{\mu}+\De_{\bar{\mu}})\a,\a\rangle_{L^{2}}\leq C(n)\|Rm(g)\|_{C^{0}}\|\a\|^{2}.$$
If the conclusion do not holds for a large constant $c(n)=cC(n)\|Rm(g)\|_{C^{0}}$, then we have
$$\la_{2k+1}\|\a\|^{2}\geq c(n)\|Rm(g)\|_{C^{0}}\|\a\|^{2}\geq c\langle(\De_{\mu}+\De_{\bar{\mu}})\a,\a\rangle_{L^{2}} ,\quad \forall\a\in(\oplus_{p+q=2k+1}\mathcal{H}^{p,q}_{d})^{\perp}.$$
where $c$ is the constant in Theorem \ref{T1}. Thus $X\in\mathcal{M}(k,c)$. From  Theorem \ref{T1}, one can see that the Betti number $b^{2k+1}=\sum_{p+q=2k+1}h^{p,q}$ is even. This contradicts our condition that $b^{2k+1}$ is odd.
\end{proof}
\subsection{Complex-$C^{\infty}$-pure and full almost complex structure}
For a complex $2$-form, we have a split as follows:
$$\Om^{2}(X)=\Om^{2,0}(X)\oplus\Om^{1,1}(X)\oplus\Om^{0,2}(X).$$
\begin{definition}(\cite[Definition 2.9]{DLZ})
	Let $H^{p,q}$ be the subspace of the complexified de Rham cohomology $H^{2}(X,\mathbb{C})$, consisting of classes which can be represented by a complex closed form of type $(p, q)$.
\end{definition}

\begin{definition}(\cite{AT,DLZ2})
	An almost-complex structure $J$ on a differential manifold $X$ is called\\
	(1) complex-$C^{\infty}$-pure if
	$$H^{2,0}(X)+H^{1,1}(X)+H^{0,2}(X)$$
	is direct;\\
	(2) complex-$C^{\infty}$-full if
	$$H^{2}_{dR}(X)=H^{2,0}(X)+H^{1,1}(X)+H^{0,2}(X)$$
	holds;\\
	(3) complex-$C^{\infty}$-pure and full if it is complex-$C^{\infty}$-pure and complex-$C^{\infty}$-full, i.e. 
	$$H^{2}_{dR}(X)=H^{2,0}(X)\oplus H^{1,1}(X)\oplus H^{0,2}(X)$$
	holds.
\end{definition}
\begin{proposition}\label{P5}
Let $(X,J,\w)$ be a closed $2n$-dimensional, ($n\geq2$), almost K\"{a}hler manifold. Then $J$ is complex-$C^{\infty}$-pure.
\end{proposition}
\begin{proof}
By the definition of complex-$C^{\infty}$-pure, we need to prove $H^{0,2}\cap H^{1,1}=\{[0]\}$, $H^{2,0}\cap H^{1,1}=\{[0]\}$ and $H^{0,2}\cap H^{2,0}=\{[0]\}$. 

Here we only provide a detailed proof of the case $H^{0,2}\cap H^{2,0}=\{[0]\}$. The method of proof in other cases is also similar.  Now let $[\a]\in H^{0,2}(X)\cap H^{2,0}(X)$. Then there exist $\a^{0,2}\in H^{0,2}$ and  $\a^{2,0}\in H^{2,0}$ such that
$$\a=\a^{0,2}+\b_{1}=\a^{2,0}+\b_{2},$$
where $\b_{1},\b_{2}$ are $d$(exact) $2$-form. Then
$$\a^{0,2}=\a^{2,0}+d\gamma,$$
where $\gamma$ is a $1$-form. Note that (see \cite[Proposition 1.2.31]{Huy})
$$\ast\a^{0,2}=\frac{1}{(n-2)!}L^{n-2}\a^{0,2}.$$
Therefore,
\begin{equation*}
\begin{split}
\|\a^{0,2}\|^{2}&=\int_{X}\a^{0,2}\wedge\ast\bar{\a}^{0,2}\\
&=\frac{1}{(n-2)!}\int_{X}\a^{0,2}\wedge\bar{\a}^{0,2}\wedge\w^{n-2}\\
&=\frac{1}{(n-2)!}\int_{X}(\a^{2,0}+d\gamma)\wedge\bar{\a}^{0,2}\wedge\w^{n-2}\\
&=\frac{1}{(n-2)!}\int_{X}\a^{2,0}\wedge\bar{\a}^{0,2}\wedge\w^{n-2}
+\frac{1}{(n-2)!}\int_{X}d(\gamma\wedge\bar{\a}^{0,2}\wedge\w^{n-2})\\
&=0.
\end{split}
\end{equation*}
Here we use the fact $d\bar{\a}^{0,2}=0$ and $d\w=0$. Hence $[\a]=0$, i.e., $H^{0,2}\cap H^{2,0}=\{[0]\}$.
\end{proof}
\begin{corollary}\label{C1}
	Let $(X,J,\w)$ be a closed $2n$-dimensional, ($n\geq2$), almost K\"{a}hler manifold. There exists an uniform positive constant $c$ such that if $X\in\mathcal{M}(2,c)$,
	then $J$ is complex $C^{\infty}$-pure and full. 	
\end{corollary}
\begin{proof}
	It's easy to see 
	$$\mathcal{H}^{p,q}_{d}\subset H^{p,2-p}\subset H_{dR}^{2}, for\ p=0,1,2.$$
	By Proposition \ref{P5}, we have 
	$$\bigoplus_{p+q=2}H^{p,q}\subset H_{dR}^{2}.$$
From Theorem \ref{T1}, we also have
	$$\sum_{p=0,1,2}\dim H^{p,2-p}\leq\dim H^{k}_{d}=\dim\mathcal{H}^{k}_{d}=\sum_{p=0,1,2}\dim\mathcal{H}_{d}^{p,2-p}\leq\sum_{p=0,1,2}\dim H^{p,2-p}$$
	Therefore, we get $$H^{2}_{dR}=H^{2,0}\oplus H^{1,1}\oplus H^{0,2},$$
	i.e., $J$ is complex $C^{\infty}$-pure and full.
\end{proof}
We also use the following notations:\\
(1) by saying that $J$ is complex-$C^{\infty}$-pure in $k$-stage we mean that the sum
$$H^{0,k}(X)+\cdots H^{p,k-p}(X)+\cdots+H^{k,0}(X)$$
is direct;\\
(2) by saying that $J$ is complex-$C^{\infty}$-full in $k$-stage we mean that  equality
$$H^{k}_{dR}(X)=H^{0,k}(X)+\cdots H^{p,k-p}(X)+\cdots+H^{k,0}(X)$$
holds;\\
(3) by saying that $J$ is complex-$C^{\infty}$-pure-and-full in $k$-stage we mean that $J$ induces the decomposition
$$H^{k}_{dR}(X)=\bigoplus_{p+q=k} H^{p,q}(X)$$
holds.
\begin{remark}
We let $\a^{p,q}\in\mathcal{H}^{p,q}_{d}$, i.e., $d\a^{p,q}=0$ and $d^{\ast}\a^{p,q}=0$. It's easy to see $\mathcal{H}^{p,q}_{d}\subset H^{p,q}$. By the Hodge decomposition, a closed $(p,q)$-form $\a^{p,q}$ can be written as $\a^{p,q}=\a_{h}+d\b$, where $\a_{h}$ is a harmonic $p+q$-form. But $\a_{h}$ may not be a pure bi-degree $(p,q)$-form.	
\end{remark}
\begin{theorem}\label{T4}
Let $(X,J,\w)$ be a closed $2n$-dimensional, ($n\geq2$), almost K\"{a}hler manifold. If
$$\mathcal{H}^{k}_{d}(X)=\bigoplus_{p+q=k}\mathcal{H}^{p,q}_{d}(X),$$
then $J$ is complex-$C^{\infty}$-pure-and-full in $k$-stage.
\end{theorem}
\begin{proof}
For any $\gamma\in\mathcal{H}^{k}_{d}$, since $\mathcal{H}^{k}_{d}=\bigoplus_{p+q=k}\mathcal{H}_{d}^{p,q}$, we get $(\De_{\bar{\pa}}+\De_{\mu})\gamma=0$. By Proposition \ref{P1}, for any $(p,q)$-form $\a^{p,q}$, we have a decomposition
$$\a^{p,q}=\a^{p,q}_{h}+(\De_{\bar{\pa}}+\De_{\mu})\b^{p,q},$$
where $\a^{p,q}_{h}$ is a harmonic form. One can check that $(\De_{\bar{\pa}}+\De_{\mu})\b^{p,q}\in(\mathcal{H}^{k}_{d})^{\bot}$ since for any $\gamma\in\mathcal{H}^{k}_{d}$,  we have
\begin{equation*}
\langle (\De_{\bar{\pa}}+\De_{\mu})\b^{p,q},\gamma\rangle_{L^{2}}=\langle \b^{p,q},(\De_{\bar{\pa}}+\De_{\mu})\gamma\rangle_{L^{2}}=0.
\end{equation*}
Since the Hodge decomposition is unique, there exist $(k-1)$-form $\eta$ and $(k+1)$-form $\delta$ such that
$$(\De_{\bar{\pa}}+\De_{\mu})\b^{p,q}=d\eta+d^{\ast}\delta.$$
Now let $[\a]\in H^{p,q}(X)$, i.e., there exists a closed $(p,q)$-form $\a^{p,q}$ such that $\a=\a^{p,q}+d(exact)$. By the above Hodge decomposition of $(p,q)$-form, there is a $\De_{d}$-harmonic $(p,q)$-form $\a_{h}^{p,q}$ such that 
$$\a^{p,q}=\a^{p,q}_{h}+d(exact)\Longrightarrow\a=\a^{p,q}_{h}+d(exact).$$
Therefore, $$H^{p,q}\cong\mathcal{H}^{p,q}_{d}.$$ 

Now we can prove $J$ is complex-$C^{\infty}$-pure in $k$-stage. Let $[\a]\in H^{i,k-i}\cap H^{j,k-j}$, $i\neq j$. Then there exist  $(i,k-i)$-form $\a_{h}^{i,k-i}\in \mathcal{H}_{d}^{i,k-i}$ and  $\a_{h}^{j,k-j}\in\mathcal{H}_{d}^{j,k-j}$ such that
$$\a=\a_{h}^{i,k-i}+d(exact)=\a_{h}^{j,k-j}+d(exact),$$
Then
$$\a_{h}^{i,k-i}=\a_{h}^{j,k-j}+d(exact).$$
Hence
\begin{equation*}
\|\a_{h}^{i,k-i}\|^{2}=\langle\a_{h}^{i,k-i},\a_{h}^{j,k-j}+d(exact)\rangle_{L^{2}}=0.
\end{equation*}
Therefore $[\a]=0$, i.e., $H^{i,k-i}\cap H^{j,k-j}=\{[0]\}$. Next, by the similar way in Corollary \ref{C1}, we can show that $H^{k}_{dR}=\bigoplus_{p+q=k} H^{p,q}$, i.e., $J$ is complex-$C^{\infty}$-pure-and-full in $k$-stage.
\end{proof}
For any $4$-dimensional almost Hermitian manifold, Cirici and Wilson proved the following result:
\begin{lemma}(\cite[Lemma 5.6]{CW1})\label{L2}
For any $4$-dimensional almost Hermitian manifold with a non-integrable almost complex structure, $\mathcal{H}^{2,0}_{d}(X)=\{0\}$. In particular, for any compact non-integrable $4$-dimensional almost K\"{a}hler manifold we have $h^{2,0}=h^{0,2}=0$.	
\end{lemma}
For any compact almost Kähler manifold of dimension $2n$, and any $p,q$ we have an orthogonal direct sum decomposition (\cite[Corollary 5.4]{CW1})
$$\mathcal{H}^{p,q}_{d}(X)=\bigoplus_{j\geq0}L^{j}(\mathcal{H}^{p-j,q-j}_{d}(X))_{prim}$$
where
$$(\mathcal{H}^{s,t}_{d}(X))_{prim}:=\mathcal{H}^{s,t}_{d}(X)\cap\ker\La.$$	
\begin{proof}[\textbf{Proof of Theorem \ref{T2}}]
	First, by Theorem \ref{T1}, we have
	$$\mathcal{H}^{2}_{d}=\mathbb{C}\w\oplus\mathcal{H}^{1,1}_{0}\oplus\mathcal{H}^{0,2}_{d}\oplus\mathcal{H}^{2,0}_{d},$$
	where 
	$$\mathcal{H}^{1,1}_{0}:=\mathcal{H}^{1,1}_{d}\cap\ker\La=\{\a\in\Om^{1,1}:\De_{d}\a=0, \La\a=0\}.$$
If $J$ is non-integrable, then following Lemma \ref{L2} (or see \cite[Lemma 2.12]{DLZ})  we have 
$$\mathcal{H}_{d}^{2,0}=\mathcal{H}_{d}^{0,2}=\{0\}.$$
Therefore, $$b^{+}_{2}=1+2h^{2,0}=1.$$ 
This contradicts our condition $b^{+}_{2}\geq2$. Hence, $J$ is integrable.
\end{proof}
\section{Hard Lefschetz Condition  on almost K\"{a}hler manifold}
\subsection{Hard Lefschetz Condition}
Recall that a symplectic manifold is said to satisfy the $dd^{\La}$-Lemma if every $d$-exact, $d^{\La}$-closed form is $dd^{\La}$-exact, namely, if $H^{\bullet}_{d+d^{\La}}\rightarrow H^{\bullet}_{dR}$ is injective. Furthermore, one says that the Hard Lefschetz Condition holds on $X$ if
(HLC) for 
$$L^{n-k}:H^{k}_{dR}(X)\rightarrow H^{2n-k}_{dR}(X),\ \forall\ 0\leq k<n$$
are isomorphisms. 

We denote $h^{k}_{d+d^{\La}}=\dim H^{k}_{d+d^{\La}}$ and $h^{k}_{dd^{\La}}=\dim H^{k}_{dd^{\La}}$. In \cite{AT15}, Angella and Tomassini introduced on a closed symplectic manifold $(X,\w)$ the following integers:
$$\De^{k}:=h^{k}_{d+d^{\La}}+h^{k}_{dd^{\La}}-2b^{k}\geq0,\ \forall k\in\mathbb{Z}$$
proving that, similarly to the complex case, their triviality characterizes the $dd^{\La}$-lemma. Using the equality $\dim H^{\bullet}_{d+d^{\La}}=\dim H^{\bullet}_{dd^{\La}}$ proved in \cite{TY1}, we can write the non-HLC degrees as follows
$$\De^{k}=2(h^{k}_{d+d^{\La}}-b^{k}),\ \forall k\in\mathbb{Z}.$$
The equalities
$$b^{k}=h^{k}_{d+d^{\La}}, \forall\ k=1,\cdots,n$$
hold on a closed symplectic $2n$-dimensional manifold if and only if it satisfies the HLC; namely the equality $b^{\bullet}=h^{\bullet}_{d+d^{\La}}$ ensures the bijectivity of the natural maps $H^{\bullet}_{d+d^{\La}}\rightarrow H^{\bullet}_{d}$, and hence the $dd^{\La}$-lemma.

We call the Hard Lefschetz Condition holds on $(\mathcal{H}^{\bullet}_{d},\w)$ if 
$$L^{n-k}:\mathcal{H}^{k}_{d}(X)\rightarrow\mathcal{H}^{2n-k}_{d}(X),\ \forall\ 0\leq k<n$$
are isomorphisms.  It is easy to see that the HLC on  $(\mathcal{H}^{\bullet}_{d},\w)$ implies the HLC on $(H^{\bullet}_{dR},\w)$ (see \cite{TW}).We will provide several equivalent conditions for HLC on $(\mathcal{H}^{\bullet}_{d},\w)$ holds on a closed almost K\"{a}hler manifold.  

Noting that $[\De_{d}+\De_{d^{\La}},L]=0$ (see \cite[Theorem 3.5]{Hua} or \cite[Theorem 5.2]{TW}). Then one can see that $\mathcal{H}^{k}_{d}=\mathcal{H}^{k}_{d^{\La}}$ implies $L^{n-k}:\mathcal{H}^{k}_{d}(X)\xrightarrow{\cong}\mathcal{H}^{2n-k}_{d}(X)$. For any $B\in\ker\La\cap\Om^{k}$, we have (see \cite[Equation (2.7)]{TY1})
$$\ast_{s}B=(-1)^{\frac{k(k+1)}{2}}\frac{1}{(n-k)!}L^{n-k}B,$$
where $\ast_{s}$ is the  symplectic star operator. Therefore, if $L^{n-k}B$ is in $\ker{\De_{d}}$, we then have
$$\De_{d}\ast_{s}B=0,$$
i.e., $d\ast_{s}B=0$ and $d\ast\ast_{s}B=0$. Therefore, $B$ is in $\ker\De_{d^{\La}}$ since $d^{\La}=(-1)^{k+1}\ast_{s}d\ast_{s}$ and $d^{\La\ast}=(-1)^{k}\ast_{s}d^{\ast}\ast_{s}$. Suppose that $L^{n-1}:\mathcal{H}^{1}_{d}(X)\xrightarrow{\cong}\mathcal{H}^{2n-1}_{d}(X)$ holds on $X$. We let $\a\in\mathcal{H}^{1}_{d}$, hence $\De_{d}(L^{n-1}\a)=0$. Therefore, $\a\in\mathcal{H}^{1}_{d^{\La}}$. It implies that  $\mathcal{H}^{1}_{d}\subset\mathcal{H}^{1}_{d^{\La}}$. Since $\dim\mathcal{H}^{1}_{d}=\dim\mathcal{H}^{1}_{d^{\La}}$, we have $\mathcal{H}^{1}_{d}=\mathcal{H}^{1}_{d^{\La}}$.
\begin{lemma}(cf. \cite[Lemma 3.5]{Hua21})\label{L6}
	For any closed almost K\"{a}hler manifold, the following identities hold:
	\begin{equation}
	\begin{split}
	&[d^{\ast},L^{l}]|_{\mathcal{H}^{\bullet}_{d}}=-lL^{l-1}d^{\La\ast}|_{\mathcal{H}^{\bullet}_{d}},\\
	&[d^{\La},L^{l}]|_{\mathcal{H}^{\bullet}_{d^{\La}}}=lL^{l-1}d|_{\mathcal{H}^{\bullet}_{d^{\La}}}.\\
	\end{split}
	\end{equation}
\end{lemma}
\begin{proof}
	The case of $l=1$: it's easy to see $[d^{\ast},L]=-d^{\La\ast}$ and $[d^{\La},L]=d$.
	We suppose that the case of $k=l-1$ it true, i.e.,
	\begin{equation*}
	[d^{\ast},L^{l-1}]|_{\mathcal{H}^{\bullet}_{d}}=-(L-1)L^{l-2}d^{\La\ast}|_{\mathcal{H}^{\bullet}_{d}},\ and\ [d^{\La},L^{l-1}]|_{\mathcal{H}^{\bullet}_{d^{\La}}}=(l-1)L^{l-2}d|_{\mathcal{H}^{\bullet}_{d^{\La}}}.
	\end{equation*}
	Thus if $k=l$: let $\a\in\mathcal{H}^{l}_{d}$, i.e., $d\a=d^{\ast}\a=0$, we then have
	\begin{equation*}
	\begin{split}
	[d^{\ast},L^{l}]\a&=[d^{\ast},L]L^{l-1}\a+Ld^{\ast}L^{l-1}\a\\
	&=[d^{\ast},L]L^{l-1}\a+L[d^{\ast},L^{l-1}]\a\\
	&=-d^{\La\ast}L^{l-1}\a-L\big{(}(l-1)L^{l-2}d^{\La\ast}\a\big{)}\\
	&=-lL^{l-1}d^{\La\ast}\a.\\
	\end{split}
	\end{equation*}
	Here we use the identity $[d^{\La\ast},L]=0$. Let $\b\in\mathcal{H}^{l}_{d^{\La}}$, i.e., $d^{\La}\b=d^{\La\ast}\b=0$, we then have
	\begin{equation*}
	\begin{split}
	[d^{\La},L^{l}]\b&=[d^{\La},L]L^{l-1}\b+Ld^{\La}L^{l-1}\b\\
	&=[d^{\La},L]L^{l-1}\b+L[d^{\La},L^{l-1}]\b\\
	&=dL^{l-1}\b+L\big{(}(l-1)L^{l-2}d\b\big{)}\\
	&=lL^{l-1}d\b.\\
	\end{split}
	\end{equation*}
	Here, we use the identity $[d, L]=0$.
\end{proof}
\begin{proposition}(cf. \cite[Theorem 5.3]{TW})\label{P8}
	Let $(X,J,\w)$ be a closed almost $2n$-dimensional K\"{a}hler manifold. For any $0\leq k<n$, then the followings statements are equivalent:\\
	(1) $\mathcal{H}^{k}_{d}(X)=\mathcal{H}^{k}_{d^{\La}}(X)$\\
	(2) $L^{n-k}:\mathcal{H}^{k}_{d}(X)\xrightarrow{\cong}\mathcal{H}^{2n-k}_{d}(X)$, \\
	(3) $L^{n-k}:\mathcal{H}^{k}_{d^{\La}}(X)\xrightarrow{\cong}\mathcal{H}^{2n-k}_{d^{\La}}(X)$,\\
	(4) $\mathcal{H}^{k}_{d}(X)\subset\mathcal{H}^{k}_{dd^{\La}}(X)$.
\end{proposition}
\begin{proof}
	$(1)\Longrightarrow(2)$, $(1)\Longrightarrow(3)$ and $(1)\Longrightarrow(4)$ : they follow from the identity $[\De_{d}+\De_{d^{\La}},L]=0$.
	
	$(2)\Longrightarrow(1)$: let $\a\in \mathcal{H}^{k}_{d}$. Since the map $L^{n-k}:\mathcal{H}^{k}_{d}\rightarrow\mathcal{H}^{2n-k}_{d}$ is an isomorphism, it follows that
	$$d\a=0,\ d^{\ast}\a=0, dL^{n-k}\a=0,\ d^{\ast}L^{n-k}\a=0.$$
	By Lemma \ref{L6}, we get $L^{n-k-1}d^{\La\ast}\a=0$. Therefore, $d^{\La\ast}\a=0$ since the map $L^{n-k-1}:\Om^{k+1}\rightarrow\Om^{2n-k-1}$ is injective. Noting that $\ast (L^{n-k}\a)$ is  also in $\mathcal{H}^{k}_{d}$, a similar argument yields  $d^{\La\ast}(\ast (L^{n-k}\a))=0$, i.e., $d^{\La} (L^{n-k}\a)=0$. Using  the identity $[d^{\ast}d+d^{\La\ast}d^{\La},L]=0$ (see \cite[Lemma 3.7]{TY1}), we obtain
	\begin{equation*}
	\begin{split}
	0&=[d^{\ast}d+d^{\La\ast}d^{\La},L^{n-k}]\a\\
	&=(d^{\ast}d+d^{\La\ast}d^{\La}) (L^{n-k}\a)-L^{n-k}(d^{\ast}d+d^{\La\ast}d^{\La})\a\\
	&=-L^{n-k}(d^{\ast}d+d^{\La\ast}d^{\La})\a.
	\end{split}
	\end{equation*}
	This gives that $(d^{\ast}d+d^{\La\ast}d^{\La})\a=0$, which implies $d^{\La}\a=0$.  Thus, we prove that $\mathcal{H}^{k}_{d}\subset\mathcal{H}^{k}_{d^{\La}}$.  Combining the fact that $\dim\mathcal{H}^{k}_{d}=\dim\mathcal{H}^{k}_{d^{\La}}$, we get $\mathcal{H}^{k}_{d}=\mathcal{H}^{k}_{d^{\La}}$.
	
	$(3)\Longrightarrow(1)$: let $\gamma\in \mathcal{H}^{k}_{d^{\La}}$. Since the map $L^{n-k}:\mathcal{H}^{k}_{d^{\La}}\rightarrow\mathcal{H}^{2n-k}_{d^{\La}}$ is an isomorphism, it follows that
	
	$$d^{\La}\gamma=0,\ d^{\La\ast}\gamma=0, d^{\La}L^{n-k}\gamma=0,\ d^{\La\ast}L^{n-k}\gamma=0.$$
	By Lemma \ref{L6}, we get $L^{n-k-1}d\gamma=0$. Therefore, $d\gamma=0$ since the map $L^{n-k-1}:\Om^{k+1}\rightarrow\Om^{2n-k-1}$ is injective. Noting that $\ast (L^{n-k}\gamma)$ is  also in $\mathcal{H}^{k}_{d^{\La}}$, a similar argument yields  $d(\ast (L^{n-k}\gamma))=0$, i.e., $d^{\ast} (L^{n-k}\gamma)=0$. Using  the identity $[dd^{\ast}+d^{\La}d^{\La\ast},L]=0$ (see \cite[Lemma 3.18]{TY1}), we obtain
	\begin{equation*}
	\begin{split}
	0&=[dd^{\ast}+d^{\La}d^{\La\ast},L^{n-k}]\gamma\\
	&=(dd^{\ast}+d^{\La}d^{\La\ast})(L^{n-k}\gamma)-L^{n-k}(dd^{\ast}+d^{\La}d^{\La\ast})\gamma\\
	&=-L^{n-k}(dd^{\ast}+d^{\La}d^{\La\ast})\gamma,\\
	\end{split}
	\end{equation*}
	This gives that $(dd^{\ast}+d^{\La}d^{\La\ast})\gamma=0$, which implies $d^{\ast}\gamma=0$.  Thus, we prove that $\mathcal{H}^{k}_{d^{\La}}\subset \mathcal{H}^{k}_{d}$.  Combining the fact that $\dim\mathcal{H}^{k}_{d}=\dim\mathcal{H}^{k}_{d^{\La}}$, we get $\mathcal{H}^{k}_{d}=\mathcal{H}^{k}_{d^{\La}}$.

	$(4)\Longrightarrow(1)$: let $\de\in \mathcal{H}^{k}_{d}$, $0<k<n$. Since $\mathcal{H}^{k}_{d}\subset\mathcal{H}^{k}_{dd^{\La}}$, it follows that
	$$d\de=0,\ d^{\ast}\de=0,\ d^{\La\ast}\de=0.$$
	By Lemma \ref{L6}, we get $d^{\ast}L^{n-k}\de=L^{n-k}d^{\ast}\de+(n-k)L^{n-k-1}d^{\La\ast}\de=0$.Therefore, $\ast(L^{n-k}\de)\in\mathcal{H}^{k}_{d}$. Also, since $\mathcal{H}^{k}_{d}\subset\mathcal{H}^{k}_{dd^{\La}}$, we have $d^{\La\ast}(\ast(L^{n-k}\de))=0$, i.e., $d^{\La}L^{n-k}\de=0$. Noting that $[d^{\ast}d+d^{\La\ast}d^{\La},L]=0$ and $[d,L]=0$, we obtain
	\begin{equation*}
	\begin{split}
	0&=[d^{\ast}d+d^{\La\ast}d^{\La},L^{n-k}]\de\\
	&=(d^{\ast}d+d^{\La\ast}d^{\La})L^{n-k}\de-L^{n-k}(d^{\ast}d+d^{\La\ast}d^{\La})\de\\
	&=-L^{n-k}(d^{\ast}d+d^{\La\ast}d^{\La})\de,\\
	\end{split}
	\end{equation*}
	Hence, $(d^{\ast}d+d^{\La\ast}d^{\La})\de=0$ since the map $L^{n-k}$ is injective. It implies that $d^{\La}\de=0$. Thus, we prove that $\mathcal{H}^{k}_{d}\subset\mathcal{H}^{k}_{d^{\La}}$.  Combining the fact that $\dim\mathcal{H}^{k}_{d}=\dim\mathcal{H}^{k}_{d^{\La}}$, we get $\mathcal{H}^{k}_{d}=\mathcal{H}^{k}_{d^{\La}}$.
\end{proof}
\begin{corollary}\label{C2}
	Let $(X,J,\w)$ be a closed  $2n$-dimensional almost K\"{a}hler manifold. If HLC on $(\mathcal{H}^{\bullet}_{d}(X),\w)$, then  $$\mathcal{H}^{k}_{d+d^{\La}}(X)=\mathcal{H}^{k}_{d}(X),\ \forall\ 0\leq k<n.$$
\end{corollary}
\begin{proof}
	Since HLC holds for $(\mathcal{H}^{\bullet}_{d},\w)$, by Proposition \ref{P8}, we have $\mathcal{H}^{k}_{d}=\mathcal{H}^{k}_{d^{\La}}$ for any $0\leq k<n$. Therefore, $\mathcal{H}^{k}_{d}\subset\mathcal{H}^{k}_{d+d^{\La}}$.  Noting that HLC  holds for  $(\mathcal{H}^{\bullet}_{d},\w)$ implies the HLC holds for $(H^{\bullet}_{dR},\w)$ (see \cite{TW}), as shown in   \cite[Corollary 3.14]{TY1}, we can get
	$\dim H^{k}_{dR}=\dim H^{k}_{d+d^{\La}}$. It implies that $\dim\mathcal{H}^{k}_{d}=\dim \mathcal{H}^{k}_{d+d^{\La}}$. We then have $\mathcal{H}^{k}_{d+d^{\La}}=\mathcal{H}^{k}_{d}$.
\end{proof}
\begin{proposition}\label{P7}
Let $(X,J,\w)$ be a closed  $4$-dimensional almost K\"{a}hler manifold.  If $\mathcal{H}^{2}_{d+d^{\La}}(X)=\mathcal{H}^{2}_{d}(X)$, then $\mathcal{H}^{2}_{d}(X)=\mathcal{H}^{2}_{d^{\La}}(X)$.  
\end{proposition}
\begin{proof}
	For any $2$-form $\a$, we have the Lefschetz decomposition 
	$$\a=B+f\w,$$ 
	where $B\in\ker\La$ and $f$ is a function.  Suppose that $\a\in\mathcal{H}^{2}_{d}$, then $d\a=0$ and $d^{\ast}\a=0$. Since $\mathcal{H}^{2}_{d+d^{\La}}=\mathcal{H}^{2}_{d}$, we also have  $d^{\La}\a=0$. Applying $d^{\ast}d+d^{\La\ast}d^{\La}$ to $\a$ and noting that 
	$[d^{\ast}d+d^{\La\ast}d^{\La},L]=0$ (see \cite{TY1}), we  get
	$$(d^{\ast}d+d^{\La\ast}d^{\La})B+L(d^{\ast}df)=0.$$
	Since $[d^{\ast}d+d^{\La\ast}d^{\La},\Lambda]=0$ (see also \cite{TY1}), by applying $\Lambda$ to the both sides, we conclude that
	$$(d^{\ast}d+d^{\La\ast}d^{\La})B=0,\ i.e.,\ dB=d^{\La}B=0$$
	and
	$$d^{\ast}df=0, i.e., f=constant.$$
	For the $2$-form $B$, we  have another decomposition
	$$B=B^{+}+B^{-},$$
	where $B^{\pm}$ is self-dual (resp. anti-self-dual) $2$-form. Furthermore $B^{\pm}$ is $J$-anti-invariant (resp. $J$-invariant). Since $d^{\ast}\a=0$, it follows that $d^{\ast}B=0$, i.e., 
	$$dB^{+}-dB^{-}=0.$$
	Therefore,
	$$d^{\La\ast}B=J^{-1}dJB=J^{-1}d(B^{-}-B^{+})=0.$$
	Hence we prove that $d^{\La\ast}\a=0$. It implies that  $\mathcal{H}^{2}_{d}\subset\mathcal{H}^{2}_{d^{\La}}$. Since $\dim\mathcal{H}^{2}_{d}=\dim\mathcal{H}^{2}_{d^{\La}}$, we have $\mathcal{H}^{2}_{d}=\mathcal{H}^{2}_{d^{\La}}$
\end{proof}
\begin{remark}
	According to Angella-Tomassini's ideas, we expect to obtain  the converse of  the Corollary \ref{C2}, i.e.
	the condition $\mathcal{H}^{\bullet}_{d+d^{\La}}=\mathcal{H}^{\bullet}_{d}$ implies that the HLC on $(\mathcal{H}^{\bullet}_{d},\w)$. Unfortunately,  it seems that we cannot get this conclusion. In fact, in the $4$-dimensional  case, 
	we have $\mathcal{H}^{2}_{d}(X)=\mathcal{H}^{2}_{d^{\La}}(X)$ by  Proposition \ref{P7}; but from Proposition \ref{P6}, the condition $\mathcal{H}^{1}_{d+d^{\La}}=\mathcal{H}^{1}_{d}$ is trivially valid, so we cannot obtain any new properties on the space of harmonic $1$-forms.
\end{remark}
\begin{theorem}
	Let $(X,J,\w)$ be a closed  $2n$-dimensional almost  K\"{a}hler manifold. Then the following statements are equivalent:\\
	(1) $\mathcal{H}^{k}_{dd^{\La}}(X)=\mathcal{H}^{k}_{d}(X)$, $\forall\ 0\leq k<n$,\\
	(2) HLC on $(\mathcal{H}_{d}^{\bullet}(X),\w)$.
\end{theorem}
\begin{proof}
	$(2)\Longrightarrow(1)$ By Proposition \ref{P8}, for any $0\leq k<n$, we have $\mathcal{H}^{k}_{d}=\mathcal{H}^{k}_{d^{\La}}$. Therefore, $\mathcal{H}^{k}_{d}\subset\mathcal{H}^{k}_{dd^{\La}}$. Since HLC on $(\mathcal{H}_{d}^{\bullet},\w)$ implies that  HLC on $(H_{dR}^{\bullet},\w)$, we get $\dim H^{k}_{dR}=\dim H^{k}_{dd^{\La}}$, i.e., $\dim\mathcal{H}^{k}_{d}=\dim\mathcal{H}^{k}_{dd^{\La}}$. Therefore, we have $\mathcal{H}^{k}_{dd^{\La}}=\mathcal{H}^{k}_{d}$.
	
	$(1)\Longrightarrow(2)$: it follows from Proposition \ref{P8}.
\end{proof}
\subsection{An estimate on $\ker\De_{\bar{\pa}+\mu}$}
We consider the operators $\bar{\pa}+\mu$ and $\pa+\bar{\mu}$ on the almost K\"{a}hler manifold (cf. \cite{dT,Hua,TT20}). For closed almost K\"{a}hler manifolds,  we have (cf. \cite{dT,Hua})
\begin{equation}\label{E11}
\De_{\bar{\pa}+\mu}=\De_{\pa+\bar{\mu}}=\frac{1}{4}(\De_{d}+\De_{d^{\La}}).
\end{equation}
Hence we obtain
\begin{theorem}(cf. \cite{Hua})\label{T7}
For any closed almost K\"{a}hler manifold of dimension $2n$ and for all $k$,  the following identities hold:
	$$\mathcal{H}^{k}_{d}\cap\mathcal{H}^{k}_{d^{\La}}=\ker\De_{\bar{\pa}+\mu}\cap\Om^{k}=\ker\De_{\pa+\bar{\mu}}\cap\Om^{k}.$$	
\end{theorem}
We also have
\begin{proposition}(cf. \cite{CW2})\label{P10}
Let $(X,J,\w)$ be a closed almost K\"{a}hler manifold. We have a decomposition as follows:
$$\Om^{k}(X)=\ker(\De_{\bar{\pa}+\mu})\cap\Om^{k}(X)\bigoplus \mathrm{Im}(\De_{\bar{\pa}+\mu})\cap\Om^{k}(X).$$
\end{proposition}
\begin{proposition}\label{P11}
Suppose that the closed almost K\"{a}hler manifold $X$ is in $\widetilde{\mathcal{M}}(k,c)$. Then for any $\gamma^{k}\in(\mathcal{H}^{k}_{\bar{\pa}+\mu}(X))^{\perp}$, we have
\begin{equation}\label{E12}
\langle\De_{d}\gamma^{k},\gamma^{k}\rangle_{L^{2}}\geq c_{6}(c-c_{5})\langle(\De_{\mu}+\De_{\bar{\mu}})\gamma^{k},\gamma^{k}\rangle_{L^{2}},
\end{equation}
where $c_{5}, c_{6}$ are two uniform positive constants. In particular, if $\de^{k}\in(\mathcal{H}^{k}_{d}(X))^{\perp}$, then
\begin{equation}\label{E15}
\langle\De_{d}\de^{k},\de^{k}\rangle_{L^{2}}\geq c_{6}(c-c_{5})\langle(\De_{\mu}+\De_{\bar{\mu}})\de^{k},\de^{k}\rangle_{L^{2}}.
\end{equation} 
\end{proposition}
\begin{proof}
Following identities in Proposition \ref{P3} and \ref{P4}, the Laplacian $\De_{d}$ can be written as
	\begin{equation*}
	\begin{split}
	\De_{d}&=\De_{\bar{\pa}+\mu}+\De_{\pa+\bar{\mu}}+2\left([\bar{\mu}^{\ast},\bar{\pa}]+[\mu,\pa^{\ast}]+[\mu^{\ast},\pa]+[\bar{\mu},\bar{\pa}^{\ast}]\right)\\
&=\De_{\bar{\pa}+\mu}+\De_{\pa+\bar{\mu}}+2\left([\bar{\mu}^{\ast},\bar{\pa}+\mu]+[\mu,\pa^{\ast}+\bar{\mu}^{\ast}]+[\mu^{\ast},\pa+\bar{\mu}]+[\bar{\mu},\bar{\pa}^{\ast}+\mu^{\ast}]\right),\\
	\end{split}
	\end{equation*}
where we use the fact $[\mu,\bar{\mu}^{\ast}]=[\bar{\mu},\mu^{\ast}]=0$. The interior product of $\De_{d}\gamma^{k}$ and $\gamma^{k}$ is 
	\begin{equation*}
	\begin{split}
	\langle\De_{d}\gamma^{k},\gamma^{k}\rangle_{L^{2}}=&\langle(\De_{\bar{\pa}+\mu}+\De_{\pa+\bar{\mu}})\gamma^{k},\gamma^{k}\rangle_{L^{2}}\\
	&+2\langle[\bar{\mu}^{\ast},\bar{\pa}+\mu]\gamma^{k},\gamma^{k}\rangle_{L^{2}}+2\langle[\mu,\pa^{\ast}+\bar{\mu}^{\ast}]\gamma^{k},\gamma^{k}\rangle_{L^{2}}\\
	&+2\langle[\mu^{\ast},\pa+\bar{\mu}]\gamma^{k},\gamma^{k}\rangle_{L^{2}}+2\langle[\bar{\mu},\bar{\pa}^{\ast}+\mu^{\ast}]\gamma^{k},\gamma^{k}\rangle_{L^{2}}.\\
	\end{split}
	\end{equation*}
 From the AM-GM Inequality, we conclude that
	\begin{equation*}
	\begin{split}
	2|\langle[\bar{\mu}^{\ast},\bar{\pa}+\mu]\gamma^{k},\gamma^{k}\rangle_{L^{2}}|
	&=2|\langle(\bar{\pa}+\mu)\gamma^{k},\bar{\mu}\gamma^{k}\rangle_{L^{2}}+\langle\bar{\mu}^{\ast}\gamma^{k},(\bar{\pa}^{\ast}+\mu^{\ast})\gamma^{k}\rangle_{L^{2}}|\\
	&\leq2\|(\bar{\pa}+\mu)\gamma^{k}\|\cdot\|\mu\gamma^{k}\|+2\|\bar{\mu}^{\ast}\gamma^{k}\|\cdot\|(\bar{\pa}^{\ast}+\mu^{\ast})\gamma^{k}\|\\
	&\leq\varepsilon\langle\De_{\bar{\pa}+\mu}\gamma^{k},\gamma^{k}\rangle_{L^{2}}+\frac{1}{\varepsilon}\langle\De_{\bar{\mu}}\gamma^{k},\gamma^{k}\rangle_{L^{2}}\\
	\end{split}
	\end{equation*}
	where $\varepsilon$ is a positive constant.  Combining the above inequalities, we get
	\begin{equation*}
	\begin{split}
	\langle\De_{d}\gamma^{k},\gamma^{k}\rangle_{L^{2}}
	\geq&\langle(\De_{\bar{\pa}+\mu}+\De_{\pa+\bar{\mu}})\gamma^{k},\gamma^{k}\rangle_{L^{2}}\\
	&-2\varepsilon\langle(\De_{\bar{\pa}+\mu}+\De_{\pa+\bar{\mu}})\gamma^{k},\gamma^{k}\rangle_{L^{2}}-\frac{2}{\varepsilon}\langle(\De_{\mu}+\De_{\bar{\mu}})\gamma^{k},\gamma^{k}\rangle_{L^{2}}\\
	=&(1-2\varepsilon)\langle(\De_{\bar{\pa}+\mu}+\De_{\pa+\bar{\mu}})\gamma^{k},\gamma^{k}\rangle_{L^{2}}-\frac{2}{\varepsilon}\langle(\De_{\mu}+\De_{\bar{\mu}})\gamma^{k},\gamma^{k}\rangle_{L^{2}}\\
	\geq&((2-4\varepsilon)c-\frac{2}{\varepsilon})\langle(\De_{\mu}+\De_{\bar{\mu}})\gamma^{k},\gamma^{k}\rangle_{L^{2}}.
	\end{split}
	\end{equation*}
	Let $\varepsilon=\frac{1}{4}$, then 
	\begin{equation}\label{E13}
	\langle\De_{d}\gamma^{k},\gamma^{k}\rangle_{L^{2}}
\geq(c-8)\langle(\De_{\mu}+\De_{\bar{\mu}})\gamma^{k},\gamma^{k}\rangle_{L^{2}}.
\end{equation}
Noting that $(\mathcal{H}^{k}_{d}(X))^{\perp}\subset(\mathcal{H}^{k}_{d}(X)\cap\mathcal{H}^{k}_{d^{\La}}(X))^{\perp}$. Hence \ref{E15} follows from (\ref{E12})
\end{proof}
\begin{proof}[\textbf{Proof of Theorem \ref{T8}}]
Let $\a$ be a $\De_{d}$-harmonic $k$-form i.e., $\a\in \mathcal{H}^{k}_{d}$. From Proposition \ref{P10}, we have the following decompositions
$$\a=\tilde{\a}_{h}+\b,$$
where $\tilde{\a}_{h}\in\mathcal{H}^{k}_{d}\cap\mathcal{H}^{k}_{d^{\La}}$ and $\b:=(\mathcal{H}^{k}_{d}\cap\mathcal{H}^{k}_{d^{\La}})^{\perp}$.
Noting that $\De_{d}\b=0$, by Proposition \ref{P11}, we  obtain
$$(\De_{\bar{\mu}}+\De_{\mu})\b=0.$$
Then by Lemma \ref{L4}, $\b$ is also in $\oplus_{p+q=k}\mathcal{H}^{p,q}_{d}\subset\mathcal{H}^{k}_{d}\cap\mathcal{H}^{k}_{d^{\La}}$. Hence, $\b=0$. It implies that $\a=\tilde{\a}_{h}$ i.e., $\mathcal{H}^{k}_{d}\subset\mathcal{H}^{k}_{d}\cap\mathcal{H}^{k}_{d^{\La}}$. Therefore, $\mathcal{H}^{k}_{d}=\mathcal{H}^{k}_{d^{\La}}$.
\end{proof}
\section{Remarks and examples}
\subsection{Almost K\"{a}hler manifold in $\widetilde{\mathcal{M}}(k,\tilde{c})$ }
 Recall that  $\widetilde{\mathcal{M}}(k,\tilde{c})$ denotes the family of closed almost K\"{a}hler manifolds with 
$$(\De_{\bar{\pa}+\mu}\a,\a)_{L^{2}}\geq \tilde{c}\langle(\De_{\mu}+\De_{\bar{\mu}})\a,\a\rangle_{L^{2}}$$
for any $\a\in(\mathcal{H}^{k}_{\bar{\pa}+\mu}(X))^{\bot}$. 
\begin{lemma}
(cf. \cite[Lemma 3.8 and Corollary 3.9]{dT})\label{L1}
Let $(X,J,\w)$ be a closed $2n$-dimensional almost K\"{a}hler manifold. If $\a\in\mathcal{H}^{k}_{d}(X)$ is a harmonic $k$-form with $0<k\leq n$, then 
\begin{equation}\label{E16}
\|d^{\La}\a\|^{2}+\|d^{\La\ast}\a\|^{2}=4\mathfrak{Re}(J^{-1}\langle\mu+\bar{\mu})J\a,d^{\La\ast}\a\rangle_{L^{2}}+4\mathfrak{Re}(J^{-1}\langle\mu^{\ast}+\bar{\mu}^{\ast})J\a,d^{\La}\a\rangle_{L^{2}}.
\end{equation}
In particular, for $k=1$, we have
\begin{equation}\label{E17}
\|d^{\La\ast}\a\|^{2}=8(\|\mu\a^{0,1}\|^{2}+\|\bar{\mu}\a^{1,0}\|^{2}).
\end{equation}
\end{lemma}
Based on \eqref{E16}, we can prove Theorem \ref{T8} in a different papproach, which also provides  a precise description on the constant $\tilde{c}$.
\begin{theorem}\label{T5}
Let $(X,J,\w)$ be a closed $2n$-dimensional almost K\"{a}hler manifold, let $0<k\leq n$. There exists a positive constant $\tilde{c}$ such that if $X\in\widetilde{\mathcal{M}}(k,\tilde{c})$, then
$$\mathcal{H}^{k}_{d}(X)=\mathcal{H}^{k}_{d^{\La}}(X),$$
where the constant $\tilde{c}$ satisfies $\tilde{c}>2$ when $k=1$ and $\tilde{c}>4$ when $k\geq2$.
\end{theorem}
\begin{proof}
If not,  we can obtain that $\mathcal{H}^{k}_{d}(X)\cap \mathcal{H}^{k}_{d^{\La}}(X)\subsetneqq \mathcal{H}^{k}_{d}(X)$. Since
$$\Om^{k}(X)=(\mathcal{H}^{k}_{d}(X)\cap\mathcal{H}^{k}_{d^{\La}}(X))\oplus(\mathcal{H}^{k}_{d}(X)\cap\mathcal{H}^{k}_{d^{\La}}(X))^{\perp}=\mathcal{H}^{k}_{d}(X)\oplus(\mathcal{H}^{k}_{d}(X))^{\perp},$$
there exists a $k$-form $\a$ such that $\a\in\mathcal{H}^{k}_{d}$, $\a\in(\mathcal{H}^{k}_{d}(X)\cap\mathcal{H}^{k}_{d^{\La}}(X))^{\perp}$ and $\a\notin\mathcal{H}^{k}_{d^{\La}}$. We denote $\a_{J}=J\a$. For $k\geq2$,  the identity (\ref{E16}) in Lemma \ref{L1} gives
\begin{equation*}
\begin{split}
\|d\a_{J}\|^{2}+\|d^{\ast}\a_{J}\|^{2}&=\|d^{\La}\a\|^{2}+\|d^{\La\ast}\a\|^{2}\\
&=
4\mathfrak{Re}\langle J^{-1}(\mu+\bar{\mu})\a_{J},d^{\La\ast}\a\rangle_{L^2}+4\mathfrak{Re}\langle J^{-1}(\mu^{\ast}+\bar{\mu}^{\ast})\a_{J},d^{\La}\a\rangle_{L^2}\\
&\leq 4\|(\mu+\bar{\mu})J\a\|\|d^{\La\ast}\a\|+4\|(\mu^{\ast}+\bar{\mu}^{\ast})J\a\|\|d^{\La}\a\|\\
&\leq 4\langle\De_{\mu+\bar{\mu}}\a_{J},\a_{J}\rangle_{L^2}^{\frac{1}{2}}(\|d^{\La\ast}\a||^{2}+\|d^{\La}\a\|^{2})^{\frac{1}{2}}.\\
\end{split}
\end{equation*}
where the last inequality follows from the Cauchy-Schwarz inequality. 
Hence 
\begin{equation*}
4\tilde{c}\langle(\De_{\mu}+\De_{\bar{\mu}})\a_{J},\a_{J}\rangle\leq \|d\a_{J}\|^{2}+\|d^{\ast}\a_{J}\|^{2}\leq 16\langle(\De_{\mu}+\De_{\bar{\mu}})\a_{J},\a_{J}\rangle.
\end{equation*}
Therefore, $\De_{\mu}\a_{J}=\De_{\bar{\mu}}\a_{J}=0$ and $d^{\La}\a=d^{\La\ast}\a=0$. This implies that $\a\in\mathcal{H}^{k}_{d^{\La}}$ contradicts our previous assumption on $\a$.  When $k=1$,  repeating the preceding argument combined with identity (\ref{E17}) yields the conclusion $\mathcal{H}^{1}_{d}(X)=\mathcal{H}^{1}_{d^{\La}}(X)$ with $\tilde{c}>2$.
\end{proof}
\begin{lemma}\label{L8}
	Let $(X,J,\w)$ be a closed almost K\"{a}hler manifold and $\a$ be a $1$-form. Then $$\mathcal{H}^{1}_{d}(X)\cap\mathcal{H}^{1}_{d^{\La}}(X)= \mathcal{H}^{1}_{\pa}(X)\cap\mathcal{H}^{1}_{\bar{\pa}}(X)\cap\mathcal{H}^{1}_{\mu}(X)\cap\mathcal{H}^{1}_{\bar{\mu}}(X).$$
\end{lemma}
Next, we  provide an example to demonstrate that the constants  in  Theorem \ref{T5} is optimal. 
\begin{example}\label{ex54}
This example  appeared in Bartolomeis and Tomassini's article (see \cite[Section 5]{dT}). Let $X=G/\Gamma$, where
\[
G=\left\{\begin{pmatrix}
1 & x & z & 0 & 0\\
0 & 1 & y & 0 & 0\\
0 & 0 & 1 & 0 & 0\\
0 & 0 & 0 & 1 & t\\
0 & 0 & 0 & 0 & 1\\
\end{pmatrix}|
x,y,z,t\in\mathbb{R}
\right\}
\]
and $\Ga$ is the discrete closed subgroup of $G$ of the matrices with integral entries, $X$ is a compact parallelizable $4$-manifold, $\Ga/[\Ga,\Ga]=\mathbb{Z}^{3}$, and so $b_{1}(X)=b_{3}(X)=3$. Consider the following four $G$-invariant $1$-forms on $\mathbb{R}^{4}$, viewed as elements in $\La^{1}(X)$:
$$\a_{1}=dx,\quad \a_{2}=dz-xdy,\quad \a_{3}=dt,\quad \a_{4}=dy,$$
where the dual elements
$$\xi_{1}=\frac{\pa}{\pa x},\quad \xi_{2}=\frac{\pa}{\pa z},\quad \xi_{3}=\frac{\pa}{\pa t},\quad \xi_{4}=\frac{\pa}{\pa y}+x\frac{\pa}{\pa z}.$$
Then 
$$d\a_{1}=d\a_{3}=d\a_{4}=0,\quad d\a_{2}=-\a_{2}\wedge\a_{4}$$
and $[\xi_{1},\xi_{4}]=\xi_{2}$, and other brackets being zero. Let $\w=\a_{3}\wedge\a_{1}+\a_{4}\wedge\a_{2}$; then $\w$ is closed and nondegenerate, and so it determines a symplectic structure on $X$; $J$ defined as $J\xi_{1}=\xi_{3}$, $J\xi_{2}=\xi_{4}$ is an almost complex structure calibrated by $\w$, i.e., $(X,J,\w)$ is an almost K\"{a}hler manifold. $\{\a_{1}, \a_{2}, \a_{3}, \a_{4}\}$ and $\{\xi_{1},\xi_{2},\xi_{3},\xi_{4}\}$ are orthonormal frames of $T^{\ast}X$ and $TX$, respectively. Moreover,
\begin{equation*}
N_{J}(\xi_{1},\xi_{2})=-\xi_{4},
\end{equation*}
and
\begin{equation*}
\begin{split}
&(\mu+\bar{\mu})\a_{1}=(\mu+\bar{\mu})\a_{3}=0,\\
&(\mu^{\ast}+\bar{\mu}^{\ast})\a_{1}=(\mu^{\ast}+\bar{\mu}^{\ast})\a_{3}=0,\\
&(\mu+\bar{\mu})\a_{2}=\frac{1}{4}(\a_{2}\wedge\a_{3}-\a_{1}\wedge\a_{4}),\\
&(\mu+\bar{\mu})\a_{4}=\frac{1}{4}(\a_{3}\wedge\a_{4}-\a_{1}\wedge\a_{2}).\\
\end{split}
\end{equation*}
Consequently, $\|\mu\|=\frac{1}{4}$, $\a_{1},\a_{3},\a_{4}$ are linearly independent and harmonic, and so 
$$\mathcal{H}^{1}_{d}(X)={\rm{span}}\{\a_{1},\a_{3},\a_{4} \}.$$
Note also that $\De_{d}\a_{2}=\a_{2}$. Furthermore, as shown in \cite[Lemma A.1]{dT}, the minimal positive eigenvalue for $\De_{d}$ acting on $1$-forms is $1$.
\begin{proposition}
Suppose that  $(X,J,\w)$ is a compact Thurston $4$-manifold as in Example \ref{ex54}. Then $X\in \widetilde{\mathcal{M}}(1,2)$.
\end{proposition}
\begin{proof}
By Lemma \ref{L4} and  \ref{L8}, one can check that $\a_{1}$ and $\a_{3}$ are in $\mathcal{H}^{1}_{d}(X)\cap\mathcal{H}^{1}_{d^{\La}}(X)$, and so
 $$\mathcal{H}^{1}_{\bar{\pa}+\mu}(X)=\mathcal{H}^{1}_{d}(X)\cap\mathcal{H}^{1}_{d^{\La}}(X)={\rm{span}}\{\a_{1},\a_{3} \}.$$ 
Here we also use Lemma \ref{T7}. To verify $(\De_{\bar{\pa}+\mu}\a,\a)_{L^{2}}\geq \tilde{c}\langle(\De_{\mu}+\De_{\bar{\mu}})\a,\a\rangle_{L^{2}}$ for any $\a\in(\mathcal{H}^{1}_{\bar{\pa}+\mu}(X))^{\bot}$, we now begin a case by case discussion. 
 
 The first case: if $\a\in(\mathcal{H}^{1}_{\bar{\pa}+\mu}(X))^{\bot}\cap(\mathcal{H}^{1}_{d}(X))^{\perp}$, noting that  the minimal positive eigenvalue for $\De_{d}$ acting on $1$-forms is $1$, then we have
$$\langle\De_{\bar{\pa}+\mu}\a,\a\rangle_{L^{2}}=\frac{1}{4}\langle\De_{d}\a,\a\rangle_{L^{2}}+\frac{1}{4}\langle\De_{d^{\La}}\a,\a\rangle_{L^{2}}\geq\frac{1}{4}\|\a\|^{2}+\frac{1}{4}\langle\De_{d^{\La}}\a,\a\rangle_{L^{2}}.$$
When $\De_{d^{\La}}\a\neq0$, it is obvious that
$$\langle\De_{\bar{\pa}+\mu}\a,\a\rangle_{L^{2}}>\frac{1}{4}\|\a\|^{2}.$$
When $\De_{d^{\La}}\a=0$, it follows that $J\a\in\ker\De_{d}={\rm{span}}\{\a_{1},\a_{3},\a_{4}\}$. Then up to a constant, we can get $J\a=\a_{4}$, and so $\a=-\a_{2}$. Therefore, 
$$\langle\De_{\bar{\pa}+\mu}\a,\a\rangle_{L^{2}}=\frac{1}{4}\|\a\|^{2}=2\langle(\De_{\mu}+\De_{\bar{\mu}})\a,\a\rangle_{L^{2}}.$$

The second case: if $\a\in(\mathcal{H}^{1}_{\bar{\pa}+\mu}(X))^{\bot}$ and $\a\in\mathcal{H}^{1}_{d}(X)$, then up to a constant, we have $\a=\a_{4}=J\a_2$. Since 
$$\De_{d^{\La}}\a_{4}=J^{-1}\De_{d}(J\a_{4})=-J^{-1}\De_{d}\a_{2}=-J^{-1}\a_{2}=\a_{4},$$
where we use the fact $J^{2}=(-1)^{k}$ for its action on differential $k$-forms. It follows that
$$\langle\De_{\bar{\pa}+\mu}\a,\a\rangle_{L^{2}}=\frac{1}{4}\|\a\|^{2}=2\langle(\De_{\mu}+\De_{\bar{\mu}})\a,\a\rangle_{L^{2}}.$$
Combining the preceding arguments, it yields that $X\in\widetilde{\mathcal{M}}(1,2)$.
\end{proof}
\end{example}
	\subsection{Almost K\"ahler manifolds in $\mathcal{M}(k,c)$}
In this section, we provide a precise description on the constant $c$, which guarantees that some almost K\"ahler manifold is contained in $\mathcal{M}(k,c)$.
\begin{proposition}\label{P13}
	Let $(X,J,\w)$ be a closed almost K\"ahler manifold. Then there exists a constant $c$ with $c>\frac{1}{2}$ such that $X\in \mathcal{M}(1,c)$.
\end{proposition}
\begin{proof}
	As shown in Equation (\ref{E1}),  we know  $c$ is a constant that obeys $c\geq\frac{1}{2}$, so it suffices to   prove that $c\neq\frac{1}{2}$, i.e., we need to show that 
	$$\langle(\De_{\bar{\pa}}+\De_{\mu})\a,\a\rangle_{L^{2}}>\frac{1}{2}\langle (\De_{\mu}+\De_{\bar{\mu}})\a,\a\rangle_{L^{2}}$$
	for any  nonzero $1$-form $\a\in(\bigoplus_{p+q=1}\mathcal{H}_{d}^{p,q})^{\bot}$. If not, there exists a nonzero $1$-form $\a\in(\bigoplus_{p+q=1}\mathcal{H}_{d}^{p,q})^{\bot}$ such that
	$$\langle(\De_{\bar{\pa}}+\De_{\mu})\a,\a\rangle_{L^{2}}=\frac{1}{2}\langle (\De_{\mu}+\De_{\bar{\mu}})\a,\a\rangle_{L^{2}}.$$
	Since $\De_{\bar{\pa}}+\De_{\mu}=\De_{\pa}+\De_{\bar{\mu}}$, we have
	$$\langle(\De_{\bar{\pa}}+\De_{\mu}+\De_{\pa}+\De_{\bar{\mu}})\a,\a \rangle_{L^{2}}=\langle (\De_{\mu}+\De_{\bar{\mu}})\a,\a\rangle_{L^{2}}.$$
	Hence,
	$$\langle(\De_{\bar{\pa}}+\De_{\pa})\a,\a \rangle_{L^{2}}=0,\quad \Longrightarrow\quad \De_{\bar{\pa}}\a=\De_{\pa}\a=0.$$
	We can write the $1$-form $\a$ as  $\a=\a^{0,1}+\a^{1,0}$, where $\a^{p,q}\in\Om^{p,q}$, then
	$$\De_{\bar{\pa}}\a^{0,1}=\De_{\pa}\a^{0,1}=0\quad \text{and}\quad \De_{\bar{\pa}}\a^{1,0}=\De_{\pa}\a^{1,0}=0.$$
	Since $\mathcal{H}_{\bar{\pa}}^{1,0}=\mathcal{H}_{\bar{\pa}}^{1,0}\cap\mathcal{H}_{\mu}^{1,0}=\mathcal{H}_{d}^{1,0}$ and $\mathcal{H}_{\pa}^{0,1}=\mathcal{H}_{\pa}^{0,1}\cap\mathcal{H}_{\bar{\mu}}^{0,1}=\mathcal{H}_{d}^{0,1}$ (see \cite[Corollary 4.6]{CW1} or Lemma \ref{L8}), we have
	$$\a^{0,1}\in\mathcal{H}_{d}^{0,1}\quad \text{and}\quad \a^{1,0}\in\mathcal{H}_{d}^{1,0}.$$
	But we assume that $\a\in(\bigoplus_{p+q=1}\mathcal{H}_{d}^{p,q})^{\bot}$, thus
	$$\a\equiv0.$$
	This contradicts our initial assumption of the 1-form $\a$.
\end{proof}
\begin{proposition}\label{P14}
Let $(X,J,\w)$ be a closed $2n$-dimensional almost K\"{a}hler manifold. Then
\begin{equation*}
(\De_{\bar{\pa}}+\De_{\mu})|_{\Om^{1}(X)}=\frac{1}{4}(\De_{d}+\De_{d^{\La}})|_{\Om^{1}(X)}.
\end{equation*}
\end{proposition}
\begin{proof}
Let $\a$ be a $1$-form on $X$. By the third identity on Proposition \ref{P4}, we have
$$\De_{d}\a=2(\De_{\bar{\pa}}+\De_{\mu}+[\bar{\mu},\pa^{\ast}]+[\mu,\bar{\pa}^{\ast}]+[\pa,\bar{\pa}^{\ast}]+[\bar{\pa},\pa^{\ast}])\a=2(\De_{\bar{\pa}}+\De_{\mu}+[\pa,\bar{\pa}^{\ast}]+[\bar{\pa},\pa^{\ast}])\a.$$
Here we use the fact that the bi-degree of operators $|[\bar{\mu},\pa^{\ast}]|=(-2,2)$, $|[\mu,\bar{\pa}^{\ast}]|=(2,-2)$. From the Lemma 3.6 and  3.7 in \cite{dT}, we also have 
$$\De_{d}-\De_{d^{\La}}=4[\pa,\bar{\pa}^{\ast}]+4[\bar{\pa},\pa^{\ast}].$$
Therefore,
$$
(\De_{\bar{\pa}}+\De_{\mu})\a=\frac{1}{4}(\De_{d}+\De_{d^{\La}})\a.$$
\end{proof}
\begin{corollary}
	Let $(X,J,\w)$ be a closed $2n$-dimensional almost K\"{a}hler manifold. If $X\in\mathcal{M}(1,c)$, where $c>2$ is a positive constant, then
	$$\mathcal{H}^{1}_{d}(X)=\mathcal{H}^{0,1}_{d}(X)\oplus\mathcal{H}^{1,0}_{d}(X).$$
\end{corollary}
\begin{proof}
Under the setting, for any $\a\in(\mathcal{H}^{1}_{\pa}(X)\cap\mathcal{H}^{1}_{\bar{\pa}}(X)\cap\mathcal{H}^{1}_{\mu}(X)\cap\mathcal{H}^{1}_{\bar{\mu}}(X))^\perp=(\mathcal{H}^{1}_{d}(X)\cap\mathcal{H}^{1}_{d^{\La}}(X))^{\perp}$, we have
$$(\De_{\bar{\pa}+\mu}\a,\a)=((\De_{\bar{\pa}}+\De_{\mu})\a,\a)\geq c((\De_{\mu}+\De_{\bar{\mu}})\a,\a).$$
Hence $X\in\widetilde{\mathcal{M}}(1,c)$. The conclusion follows from Theorem \ref{T5}.
\end{proof}
\begin{proposition}
	Let $(X,J,\w)$ be a closed almost K\"ahler $4$-manifold. Then there exists a constant $c$ with $c>\frac{1}{2}$ such that $X\in \mathcal{M}(2,c)$.
\end{proposition}
\begin{proof}
	We also only need to prove that $c\neq \frac{1}{2}$. 
	If not, there exists a nonzero $2$-form $\a\in(\bigoplus_{p+q=2}\mathcal{H}_{d}^{p,q})^{\bot}$ such that
	$$\langle(\De_{\bar{\pa}}+\De_{\mu})\a,\a\rangle_{L^{2}}=\frac{1}{2}\langle (\De_{\mu}+\De_{\bar{\mu}})\a,\a\rangle_{L^{2}}.$$
	Similarly, as in the proof of  Proposition \ref{P13}, we get 
	$$ \De_{\bar{\pa}}\a=\De_{\pa}\a=0.$$
	Noting that $\a=\a^{2,0}+\a^{1,1}+\a^{0,2}$, where $\a^{p,q}\in\Om^{p,q}$, we conclude that
	$$ \De_{\bar{\pa}}\a^{p,q}=\De_{\pa}\a^{p,q}=0,\quad (p+q=2).$$
	Since the dimension of $X$ is $4$, we have
	\begin{align*}
	\mathcal{H}_{\pa}^{0,2}=\mathcal{H}_{\pa}^{0,2}\cap\mathcal{H}_{\bar{\mu}}^{0,2}=\mathcal{H}_{d}^{0,2},\\
	\mathcal{H}_{\bar{\pa}}^{2,0}=\mathcal{H}_{\bar{\pa}}^{2,0}\cap\mathcal{H}_{\mu}^{2,0}=\mathcal{H}_{d}^{2,0},\\
	\mathcal{H}_{\bar{\pa}}^{1,1}=\mathcal{H}_{\bar{\pa}}^{1,1}\cap\mathcal{H}_{\mu}^{1,1}=\mathcal{H}_{d}^{1,1}.
	\end{align*}
	Hence,
	$$\a^{0,2}\in\mathcal{H}_{d}^{0,2},\quad  \a^{1,1}\in\mathcal{H}_{d}^{1,1}\quad  \text{and}\quad \a^{0,2}\in\mathcal{H}_{d}^{0,2}.$$
	But we assume that $\a\in(\bigoplus_{p+q=2}\mathcal{H}_{d}^{p,q})^{\bot}$, thus
	$$\a\equiv 0.$$
	This contradicts our initial assumption of the 2-form $\a$.
\end{proof}

\section*{Acknowledgements}
We would like to thank S.O. Wilson for kind comments regarding his article \cite{CW1} and  D.X. Lin for the helpful discussions. This work is supported by the National Natural Science Foundation of China (No. 12271496) and the Youth Innovation Promotion Association CAS, the Fundamental Research Funds of the Central Universities, and the USTC Research Funds of the Double First-Class Initiative.

\end{document}